\documentclass[12pt]{amsart}
\usepackage[margin=1in]{geometry}
\usepackage{amsmath,amsfonts,amsthm,amssymb,bbm}
\usepackage{graphicx,color,dsfont}
\usepackage{enumitem}
\usepackage{fourier}
\usepackage{amssymb}

\begin{document}

\newtheorem{theorem}{Theorem}
\newtheorem{lemma}{Lemma}
\newtheorem{proposition}{Proposition}
\newtheorem{rmk}{Remark}
\newtheorem{example}{Example}
\newtheorem{exercise}{Exercise}
\newtheorem{definition}{Definition}
\newtheorem{corollary}{Corollary}
\newtheorem{notation}{Notation}
\newtheorem{claim}{Claim}
\newtheorem{Conjecture}{Conjecture}

\newtheorem{dif}{Definition}

 \newtheorem{thm}{Theorem}[section]
 \newtheorem{cor}[thm]{Corollary}
 \newtheorem{lem}[thm]{Lemma}
 \newtheorem{prop}[thm]{Proposition}
 \theoremstyle{definition}
 \newtheorem{defn}[thm]{Definition}
 \theoremstyle{remark}
 \newtheorem{rem}[thm]{Remark}
 \newtheorem*{ex}{Example}
 \numberwithin{equation}{section}

\newcommand{\vertiii}[1]{{\left\vert\kern-0.25ex\left\vert\kern-0.25ex\left\vert #1
    \right\vert\kern-0.25ex\right\vert\kern-0.25ex\right\vert}}

\newcommand{\R}{{\mathbb R}}
\newcommand{\C}{{\mathbb C}}
\newcommand{\U}{{\mathcal U}}
\newcommand{\norm}[1]{\left\|#1\right\|}
\renewcommand{\(}{\left(}
\renewcommand{\)}{\right)}
\renewcommand{\[}{\left[}
\renewcommand{\]}{\right]}
\newcommand{\f}[2]{\frac{#1}{#2}}
\newcommand{\im}{i}
\newcommand{\cl}{{\mathcal L}}
\newcommand{\ck}{{\mathcal K}}

\newcommand{\al}{\alpha}
\newcommand{\be}{\beta}
\newcommand{\wh}[1]{\widehat{#1}}
\newcommand{\ga}{\gamma}
\newcommand{\Ga}{\Gamma}
\newcommand{\de}{\delta}
\newcommand{\ben}{\beta_n}
\newcommand{\De}{\Delta}
\newcommand{\ve}{\varepsilon}
\newcommand{\ze}{\zeta}
\newcommand{\Th}{\Theta}
\newcommand{\ka}{\kappa}
\newcommand{\la}{\lambda}
\newcommand{\laj}{\lambda_j}
\newcommand{\lak}{\lambda_k}
\newcommand{\La}{\Lambda}
\newcommand{\si}{\sigma}
\newcommand{\Si}{\Sigma}
\newcommand{\vp}{\varphi}
\newcommand{\om}{\omega}
\newcommand{\Om}{\Omega}
\newcommand{\ra}{\rightarrow}

\newcommand{\ro}{{\mathbf R}}
\newcommand{\rn}{{\mathbf R}^n}
\newcommand{\rd}{{\mathbf R}^d}
\newcommand{\rmm}{{\mathbf R}^m}
\newcommand{\rone}{\mathbf R}
\newcommand{\rtwo}{\mathbf R^2}
\newcommand{\rthree}{\mathbf R^3}
\newcommand{\rfour}{\mathbf R^4}
\newcommand{\ronen}{{\mathbf R}^{n+1}}
\newcommand{\ku}{\mathbf u}
\newcommand{\kw}{\mathbf w}
\newcommand{\kf}{\mathbf f}
\newcommand{\kz}{\mathbf z}
\newcommand{\bb}{\mathbf B}

\newcommand{\N}{\mathbf N}

\newcommand{\tn}{\mathbf T^n}
\newcommand{\tone}{\mathbf T^1}
\newcommand{\ttwo}{\mathbf T^2}
\newcommand{\tthree}{\mathbf T^3}
\newcommand{\tfour}{\mathbf T^4}

\newcommand{\zn}{\mathbf Z^n}
\newcommand{\zp}{\mathbf Z^+}
\newcommand{\zone}{\mathbf Z^1}
\newcommand{\zz}{\mathbf Z}
\newcommand{\ztwo}{\mathbf Z^2}
\newcommand{\zthree}{\mathbf Z^3}
\newcommand{\zfour}{\mathbf Z^4}

\newcommand{\hn}{\mathbf H^n}
\newcommand{\hone}{\mathbf H^1}
\newcommand{\htwo}{\mathbf H^2}
\newcommand{\hthree}{\mathbf H^3}
\newcommand{\hfour}{\mathbf H^4}

\newcommand{\cone}{\mathbf C^1}
\newcommand{\ctwo}{\mathbf C^2}
\newcommand{\cthree}{\mathbf C^3}
\newcommand{\cfour}{\mathbf C^4}
\newcommand{\dpr}[2]{\langle #1,#2 \rangle}

\newcommand{\sn}{\mathbf S^{n-1}}
\newcommand{\sone}{\mathbf S^1}
\newcommand{\stwo}{\mathbf S^2}
\newcommand{\sthree}{\mathbf S^3}
\newcommand{\sfour}{\mathbf S^4}

\newcommand{\lp}{L^{p}}
\newcommand{\lppr}{L^{p'}}
\newcommand{\lqq}{L^{q}}
\newcommand{\lr}{L^{r}}
\newcommand{\echi}{(1-\chi(x/M))}
\newcommand{\chip}{\chi'(x/M)}

\newcommand{\wlp}{L^{p,\infty}}
\newcommand{\wlq}{L^{q,\infty}}
\newcommand{\wlr}{L^{r,\infty}}
\newcommand{\wlo}{L^{1,\infty}}

\newcommand{\lprn}{L^{p}(\rn)}
\newcommand{\lptn}{L^{p}(\tn)}
\newcommand{\lpzn}{L^{p}(\zn)}
\newcommand{\lpcn}{L^{p}(\cn)}
\newcommand{\lphn}{L^{p}(\cn)}

\newcommand{\lprone}{L^{p}(\rone)}
\newcommand{\lptone}{L^{p}(\tone)}
\newcommand{\lpzone}{L^{p}(\zone)}
\newcommand{\lpcone}{L^{p}(\cone)}
\newcommand{\lphone}{L^{p}(\hone)}

\newcommand{\lqrn}{L^{q}(\rn)}
\newcommand{\lqtn}{L^{q}(\tn)}
\newcommand{\lqzn}{L^{q}(\zn)}
\newcommand{\lqcn}{L^{q}(\cn)}
\newcommand{\lqhn}{L^{q}(\hn)}

\newcommand{\lo}{L^{1}}
\newcommand{\lt}{L^{2}}
\newcommand{\li}{L^{\infty}}
\newcommand{\beqn}{\begin{eqnarray*}}
\newcommand{\eeqn}{\end{eqnarray*}}
\newcommand{\pplus}{P_{Ker[\cl_+]^\perp}}

\newcommand{\co}{C^{1}}
\newcommand{\coi}{C_0^{\infty}}

\newcommand{\ca}{\mathcal A}
\newcommand{\cs}{\mathcal S}
\newcommand{\cm}{\mathcal M}
\newcommand{\cf}{\mathcal F}
\newcommand{\cb}{\mathcal B}
\newcommand{\ce}{\mathcal E}
\newcommand{\cd}{\mathcal D}
\newcommand{\cg}{\mathcal G}
\newcommand{\cn}{\mathcal N}
\newcommand{\cz}{\mathcal Z}
\newcommand{\crr}{\mathbf R}
\newcommand{\cc}{\mathcal C}
\newcommand{\ch}{\mathcal H}
\newcommand{\cq}{\mathcal Q}
\newcommand{\cp}{\mathcal P}
\newcommand{\cx}{\mathcal X}
\newcommand{\eps}{\epsilon}

\newcommand{\pv}{\textup{p.v.}\,}
\newcommand{\loc}{\textup{loc}}
\newcommand{\intl}{\int\limits}
\newcommand{\iintl}{\iint\limits}
\newcommand{\dint}{\displaystyle\int}
\newcommand{\diint}{\displaystyle\iint}
\newcommand{\dintl}{\displaystyle\intl}
\newcommand{\diintl}{\displaystyle\iintl}
\newcommand{\liml}{\lim\limits}
\newcommand{\suml}{\sum\limits}
\newcommand{\ltwo}{L^{2}}
\newcommand{\supl}{\sup\limits}
\newcommand{\df}{\displaystyle\frac}
\newcommand{\p}{\partial}
\newcommand{\Ar}{\textup{Arg}}
\newcommand{\abssigk}{\widehat{|\si_k|}}
\newcommand{\ed}{(1-\p_x^2)^{-1}}
\newcommand{\tT}{\tilde{T}}
\newcommand{\tV}{\tilde{V}}
\newcommand{\wt}{\widetilde}
\newcommand{\Qvi}{Q_{\nu,i}}
\newcommand{\sjv}{a_{j,\nu}}
\newcommand{\sj}{a_j}
\newcommand{\pvs}{P_\nu^s}
\newcommand{\pva}{P_1^s}
\newcommand{\cjk}{c_{j,k}^{m,s}}
\newcommand{\Bjsnu}{B_{j-s,\nu}}
\newcommand{\Bjs}{B_{j-s}}
\newcommand{\Ly}{L_i^y}
\newcommand{\dd}[1]{\f{\partial}{\partial #1}}
\newcommand{\czz}{Calder\'on-Zygmund}
\newcommand{\chh}{\mathcal H}

\newcommand{\lbl}{\label}
\newcommand{\beq}{\begin{equation}}
\newcommand{\eeq}{\end{equation}}
\newcommand{\beqna}{\begin{eqnarray*}}
\newcommand{\eeqna}{\end{eqnarray*}}
\newcommand{\bp}{\begin{proof}}
\newcommand{\ep}{\end{proof}}
\newcommand{\bprop}{\begin{proposition}}
\newcommand{\eprop}{\end{proposition}}
\newcommand{\bt}{\begin{theorem}}
\newcommand{\et}{\end{theorem}}
\newcommand{\bex}{\begin{Example}}
\newcommand{\eex}{\end{Example}}
\newcommand{\bc}{\begin{corollary}}
\newcommand{\ec}{\end{corollary}}
\newcommand{\bcl}{\begin{claim}}
\newcommand{\ecl}{\end{claim}}
\newcommand{\bl}{\begin{lemma}}
\newcommand{\el}{\end{lemma}}
\newcommand{\dea}{(-\De)^\be}
\newcommand{\naa}{|\nabla|^\be}
\newcommand{\cj}{{\mathcal J}}
\newcommand{\ci}{{\mathcal I}}
\newcommand{\ubb}{{\mathbf u}}

\title[Standing Waves of the Schr\"odinger equation with Concentrated 
nonlinearity]{On the standing waves of the Schr\"odinger equation with concentrated nonlinearity}

\author[Abba Ramadan and  Atanas G. Stefanov]{\sc Abba Ramadan and Atanas G. Stefanov}
\address{Department of Mathematics,
University of Kansas,
1460 Jayhawk Boulevard,  Lawrence KS 66045--7523, USA}
\email{aramadan@ku.edu}, \email{stefanov@ku.edu}

 \thanks{A. Ramadan is partially supported by   NSF-DMS,   \# 1908626.  
 A. Stefanov acknowledges  partial support  from NSF-DMS,   \# 1908626.}



\date{\today}
 
\begin{abstract}
We study the  concentrated NLS on $\rn$, with power non-linearities, driven by the fractional Laplacian, $(-\De)^s, s>\f n2$.  We construct the solitary waves  explicitly, in an optimal range of the parameters, so that they belong to the natural energy space $H^s$. Next, we provide a complete classification of  their spectral stability. Finally, we show that the waves are non-degenerate and consequently orbitally stable, whenever they are spectrally stable.

 Incidentally, our construction shows that the soliton profiles for the concentrated NLS  are in fact exact minimizers of the Sobolev embedding $H^s(\rn)\hookrightarrow L^\infty(\rn)$, which provides an alternative calculation and justification of the sharp constants in these inequalities. 
\end{abstract}

\maketitle
\section{ Introduction} 
The (focusing) nonlinear Schr\"odinger equation, with generalized power non-linearity 
\begin{equation}
\label{10} 
iu_t + \De u+  |u|^{2\si} u=0, (t,x)\in \rone\times \rn
\end{equation}
is a basic model in theoretical physics and applied mathematics (e.g.  quantum mechanics and water waves theory) and practical engineering applications. It has been studied extensively in the last fifty years, in particular with regards to the well-posedness of the Cauchy problem and the stability of its solitary waves. The well-posedness theory is classical by now, \cite{Caz} and states that local well-posedness holds for any $\si>0$, whenever the data $u_0\in H^s(\rn), s\geq 0$. The global well-posedness results rely upon the conservation law, which state 
that the following quantities, namely the mass $M(u)$ and the energy $E(u)$
\begin{eqnarray*}
M(u)&=& \int_{\rn} |u(t,x)|^2 dx=const. \\
E(u) &=& \f{1}{2}\int_{\rn} |\nabla u(t,x)|^2 dx - \f{1}{2\si +2} \int_{\rn} |u(t,x)|^{2\si+2} dx=const.
\end{eqnarray*}
 As such, solutions with initial data $u_0\in H^1(\rn)$ yield global solutions, whenever the problem is $L^2$ sub-critical, i.e.  $\si<\f{2}{n}$, while for $\si\geq \f{2}{n}$,  some initial data gives rise to finite time blow-ups. Interestingly, the ground states for \eqref{10} are stable exactly in the $L^2$ critical range $\si<\f{2}{n}$, while they are unstable in the supercritical regime $\si>\f{2}{n}$. In the $L^2$ critical case, $\si=\f{2}{n}$, the equation \eqref{10} exhibits an additional symmetry, the so-called quasi-conformal invariance, which allows one to exhibit special self-similar type solutions, which show that blows up also occurs in the critical case. 

In this work, we analyze  related model, the focusing non-linear Schr\"odinger equation with concentrated non-linearity. As our dispersive models will be driven by fractional Laplacians, let us introduce  the proper framework. We set the  Fourier transform and its inverse by the formulas 
$$
\hat{f}(\xi)=\int_{\rn} f(x) e^{-2\pi i x\cdot \xi} dx; \ \  f(x)=\int_{\rn} \hat{f}(\xi) e^{2\pi i x\cdot \xi} 
d\xi.
$$
In that case, the Laplacian is given as a Fourier multiplier (on the space of Schwartz functions $\cs$) via $\widehat{(-\De) f}=4\pi^2 |\xi|^2 \hat{f}(\xi)$. More generally, for all $s>0$ 
$$
\widehat{(-\De)^s f}=(2\pi  |\xi|)^{2s} \hat{f}(\xi). 
$$
Now, the focussing NLS with concentrated non-linearity is the following 
\begin{equation}
\label{20} 
\left\{\begin{array}{l}
iu_t =( (-\De)^s -  |u|^{2\si} \de_0) u, \ \ (t,x)\in \rone\times \rn \\ 
u(0,x)=u_0(x)
\end{array}
\right.
\end{equation}
 Our definition of a solution is as follows: a  continuous in $x$  function $u$ is a weak solution of \eqref{20}, if it  satisfies 
\begin{eqnarray*}
& & 
i \left(\dpr{u(t, \cdot)}{\psi(t, \cdot)} - \dpr{u_0}{\psi(0,\cdot)} - \int_0^t \dpr{u(s,\cdot)}{\psi_s(s,\cdot)}\right)=\\
& & = \int_0^t \dpr{(-\De)^{\f{s}{2}} u(s,\cdot)}{(-\De)^{\f{s}{2}} \psi(s,\cdot)}ds- \int_0^t |u(s,0)|^{2\si}u(s,0) \psi(s,0)ds
\end{eqnarray*}
for all test functions $\psi$. For the case of the standard Laplacian, i.e. $s=1$, the model \eqref{20} has been used to model resonant tunneling, \cite{JL}, the dynamics of mixed states, \cite{Nier}, quantum turbulence, \cite{Az}, the generation of weakly bounded states  close to the instability, \cite{Wood} among others.  

The fractional Laplacian perturbed by a delta potential, together with their self-adjoint extensions and various applications,  have been recently considered in \cite{CFT}. In the case of one spatial dimension, $n=1$ and $s>\f{1}{2}$ , the local well-posedness as well as the conservation of mass and energy 
\begin{eqnarray}
\label{conserve}
	M(u) &=& \int_{\rn} |u(t,x)|^2 dx=const. \\
	E(u) &=& \f{1}{2} \|(-\De)^{\f{s}{2}} u\|_{L^2}^2 - 
	\f{1}{2\si +2} |u(t,0)|^{2\si+2} =const.
\end{eqnarray}
was recently  established in \cite{CFT}. Even though the results in \cite{CFT} are stated for the one dimensional case only,  it seems plausible that they  can be extended in any dimension $n$ and $s>\f n2$  using similar techniques. It is important to note that since our  interests is  in continuous in $x$ functions, the natural spaces for well-posedness, in the scale of the Sobolev spaces, should be $H^s(\rn), s>\f n2$. Another reason why this is, in our opinion, a more natural class of problems to consider,  is that we would like  waves which belong to the energy space $H^s(\rn)$, as dictated by the conservation of $E(u)$. As we shall see below, the solitary waves belong to this space only for $s>\f n2$. 

It has to be noted however, that it is certainly possible (and it is in fact considerably more challenging, the furthest one is from the 
threshold $s=\f n2$) to consider \eqref{20} in cases where $s<\f n2$, and this has been addressed, at least in low dimensional situations, in the recent papers, \cite{ADFT, ADFT2, ANO, ANO2, AT}. 
Regarding analysis of blow up solutions for the concentrated NLS (although not necessarily in the case of interest $s>\f n2$), this was carried out recently  in \cite{ADFT2}. 

Our main  interest in the model \eqref{20} are its solitary waves and their stability. More specifically, we consider  solutions in the form $u=e^{i \om t} \phi$, $\phi$ real-valued,  which naturally satisfy the profile equation. This is again understood in the weak sense described above
\begin{equation}
\label{30} 
(-\De)^s \phi+\om \phi - |\phi(0)|^{2\si} \phi(0) \de_0=0.
\end{equation}
We take the opportunity to note that  in many cases considered herein, one cannot expect the positivity of  $\phi$, as in the classical case. This is why, we keep the absolute value in \eqref{30}. 

The question for the stability  of these  waves, when $s=1$,  has been considered in 
several contexts recently, see \cite{AN1}, \cite{ANO}, \cite{ANO2} for the three dimensional case $n=3$ and  \cite{ACCT}, for $n=2$. Again, some of these  works consider cases mostly outside of the range of consideration herein $s>\f n2$. 
 
Before we address the construction of the solitons (that is, solutions of \eqref{30}), and since our situation is a bit non-standard, we would like to outline the framework for the  stability of the waves. 
\subsection{Linearized problem for the concentrated NLS}
As is customary, the spectral/linearized stability of the standing waves, i.e. the solutions of \eqref{30}, guides us in the study of the actual non-linear dynamics, when one starts close   to these solutions\footnote{and indeed in the understanding of the ranges of $\si$ that give global existence viz. a viz blow up, as discussed above}.   
More specifically, if we linearize around the solitary waves and ignore quadratic and higher order contributions, we obtain a linear system, whose spectral information plays a part in the dynamics. To that end, we take $u=e^{i\om t}(\phi+v)$ and plug it in \eqref{20}, ignoring any $O(v^2)$ term, utilizing \eqref{30} and setting $ (v_1, v_2):=(\Re v, \Im v)=v $,  we obtain 
\begin{equation}
\label{412} 
\left(\begin{array}{c}
\Re v\\ \Im v
\end{array}\right)_t = \left(\begin{array}{cc}
0 & -1 \\ 1 & 0
\end{array}\right) \left(\begin{array}{cc}
\cl_- & 0 \\ 0 & \cl_+
\end{array}\right) \left(\begin{array}{c}
\Re v\\ \Im v
\end{array}\right), 
\end{equation}
where the following fractional Schr\"odinger operators are introduced
\begin{eqnarray*}
\cl_+ &=&(-\De)^s +\om - (2\sigma+1) |\phi(0)|^{2\sigma}\delta_0, \\
\cl_- &=&(-\De)^s +\om -  |\phi(0)|^{2\sigma}\delta_0.
\end{eqnarray*}
This formulas are heuristic in the sense that the operators $\cl_\pm$ are not yet properly defined, in terms of domains etc. This is generally not an easy task\footnote{Although, as it turns out, we shall need to restrict to the case $s>\f{n}{2}$, which will make such definitions in a sense canonical} will appropriately be define in later section, see Section \ref{sec:2.3}. 
With the introduction of  the operators
$$
\cj:=\left(\begin{array}{cc}
0 & -1 \\ 1 & 0
\end{array}\right) , \cl:=\left(\begin{array}{cc}
\cl_- & 0 \\ 0 & \cl_+
\end{array}\right), 
$$
and the assignment $\left(\begin{array}{c}
\Re v\\ \Im v
\end{array}\right)\to e^{\la t} \left(\begin{array}{c}
v_1\\ v_2
\end{array}\right)=:e^{\la t} \vec{v}$, we obtain the following time-independent linearized eigenvalue problem 
\begin{equation}
\label{414} 
\cj \cl \vec{v}=\la \vec{v}. 
\end{equation}
Since we are interested in stability of waves, it will be appropriate to give a standard definition of stability as follow.

\begin{definition}
	\label{defi:19} The wave $e^{i \om t} \phi$ is said to be {\bf spectrally unstable}, if the eigenvalue problem \eqref{414} has a  solution $(\la, \vec{v})$, with $\Re \la>0$ and $\vec{v}\neq 0, \vec{v} \in D(\cl)$. Otherwise, i.e. if \eqref{414} has no non-trivial solutions, with $\Re \la>0$, we say that the wave is {\bf spectrally stable}.  
	
	We say that $e^{i \om t} \phi$ is  {\bf orbitally stable} solution of \eqref{20}, if for every $\eps>0$, there exists $\de=\de(\eps)$, so that whenever 
	$\|u_0- \phi\|_{H^{s}(\rn)}<\de$, then the following statements hold.
	\begin{itemize}
		\item The solution $u$ of \eqref{20}, in appropriate sense, with initial data 
		$u_0\in H^s$ is globally in $H^{s}(\rn)$, i.e. $u(t, \cdot)\in H^{s}(\rn)$. 
	\item	
		$$\sup_{t>0} \inf_{\theta\in \rone} \|u(t, \cdot) - e^{ -i(\om t+ \theta)} \phi(\cdot)\|_{H^s(\rn)}<\eps. 
		$$
	\end{itemize}
	
\end{definition}
The connection between the two main notions of stability, namely spectral and orbital stability has been explored extensively in the literature - see for example the excellent book \cite{KP}. Generally speaking, spectral stability is a prerequisite for orbital stability, and in many cases of interest and under some natural, but not necessarily easy to check conditions, see Section 5.2.2 in \cite{KP}, spectral stability implies orbital stability. In the case under consideration, the conditions outlined in \cite{KP} do not apply, so we provide a direct proof of orbital stability via contradiction argument, in the cases of spectral stability, by following the original idea by T.E. Benjamin. 

We should also point out that the reverse connection, namely spectral instability implies orbital instability. Basic heuristics (or even some more formal arguments) imply that this must be indeed the case.  However, in terms of rigorous results, there are results,   
if there is a positive instability mode present,  via a direct ODE Lyapunov method - see for example \cite{KarS} for a sample statement. As in the stability case, there is no satisfactory general result that would cover our examples, so we leave our rigorous conclusions at the level of  spectral instability of the waves and we do not comment further on (the likely) orbital instability thereof. 

\subsection{Main results}
Before we present our  existence result for the singular elliptic problem \eqref{30},    let us introduce a function 
$\cg_s^\la$, which will be a basic building block in our analysis. Namely, for all $\la>0$ and $s>0$,  
$$
\widehat{\cg_s^\la} (\xi) = \f{1}{(2\pi|\xi|)^{2s}+\la}. 
$$
We first state a few results related to the existence of the waves $\phi_\om$, under some conditions on the parameters $s,\om,n$, which turn out to be necessary as well. Then, we discuss the fact that these waves are also minimizers of a Sobolev embedding inequality and we present its exact constant. 
\subsubsection{Existence of the waves $\phi_\om$}
\begin{theorem}(Existence standing waves of the concentrated NLS)
	\label{r:1} 
	Let $\om>0, s>\f n2$ and $\si>0$. Then,  the function $\phi$, with 
$$
	\hat{\phi}_\om(\xi)= \left(\int_{\rn} \f{1}{(2\pi|\xi|)^{2s}+\om}  d\xi\right)^{-(1+\f{1}{2\si})} \f{1}{(2\pi|\xi|)^{2s}+\om}.
	$$
	is a solution of \eqref{30}. Alternatively, 
	$$
	\phi_\om(x)= \f{\cg_s^\om(x)}{(\cg_s^\om(0))^{1+\f{1}{2\si}}}.
	$$
\end{theorem}
Interestingly, the conditions for $\om$ and $s$ in Theorem \ref{r:1} are necessary for the existence of solutions $\phi\in H^s(\rn)\cap C(\rn)$ of \eqref{30}. 
\begin{proposition}
	\label{r:2} Let $\phi\in H^s(\rn)\cap C(\rn)$ be a weak solution of \eqref{30}. Then, it must be that either $\om>0, s>\f n2$ or $\om<0, s<\f n2$. 
\end{proposition}
\noindent The proof of Proposition \ref{r:2} proceeds via the Pohozaev's identities, see Section \ref{sec:2.2} below. 

In the process of the variational construction of the waves $\phi_\om$, we establish a non-surprising connection to the problem for the optimal constant in the Sobolev embedding $H^s(\rn)\hookrightarrow L^\infty(\rn)$. More specifically, we establish that $\cg_s=\cg_s^1$ (and consequently $\phi_1$) are  $H^s$ functions that saturate the Sobolev embedding, with the optimal Sobolev constant 
\begin{equation}
\label{s:12} 
2^{n-1} \pi^{\f{n}{2}-1} \Ga\left(\f{n}{2}\right)\sin\left(\f{n\pi}{2s}\right)  \|u\|_{L^\infty}^2\leq \|(-\De)^{\f s2} u\|_{L^2}^2+ \|u\|_{L^2}^2.
\end{equation}
We formulate the result in the following proposition. 
\begin{proposition}
	\label{r:4} 
	The function $\cg_s$ is a solution to the Sobolev embedding minimization problem 
	$$
	 \inf_{u\in \cs: u\neq 0} \f{\|(-\De)^{\f{s}{2}} u\|_{L^2}^2+ 
	 	\|u\|_{L^2}^2}{\|u\|_{L^\infty}^2}= 2^{n-1} \pi^{\f{n}{2}-1} \Ga\left(\f{n}{2}\right)\sin\left(\f{n\pi}{2s}\right)
	$$
\end{proposition}
Next, we turn our attention  towards the stability results. We first state spectral stability/instability  result, followed by orbital stability statements. 
\subsubsection{Stability characterization of the waves $\phi_\om$}
\begin{theorem}
	\label{theo:12} 
	Let $n\geq 1$, $s>\f n2$ and $\om>0$. Then, the waves $e^{i \om t} \phi_\om$ are spectrally stable if and only if 
	$$
	0<\si<\f{2s}{n}-1.
	$$
	That is, the waves are stable for all   $0<\si<\f{2s}{n}-1$ and unstable, when $\si>\f{2s}{n}-1$. Moreover, the instability is due to a presence of a single and simple 
	real mode in the eigenvalue problem \eqref{414}. 
\end{theorem}
Finally, before we state our orbital stability results, we need to make some natural assumptions regarding the well-posedness of the Cauchy problem \eqref{20}. 

Clearly, the orbital stability is only expected to hold for the case $\si<\f{2s}{n}-1$, so we assume that henceforth. We make the following  \textbf{key assumptions}:
\begin{enumerate}
	\item  The solution map $g\to u_g$ has {\bf continuous dependence on initial data property in a neighborhood of $\phi$}. That is,  there exists $T_0>0$, so that for all $\eps>0$, there exists $\de>0$, so that whenever  $g: \|g- \phi\|_{H^s}<\de$, then 
	$\sup_{0<t<T_0} \|u_g(t, \cdot)- e^{- i \om t} \phi_\om\|_{H^s}<\eps$.  
	\item  {\bf All initial data, sufficiently close to $\phi_\om$ in $H^s$ norm, generates a global in time solution} $u_g$ of \eqref{20}. 
In addition,  the $L^2$ norm and the Hamiltonian for these solutions are conserved. That is 
	$$
	M[u_g(t)]=M[g], E[u_g(t)]=E[g].
	$$
\end{enumerate}
First, let us mention that this exact  result is already available in the one dimensional case $n=1$, \cite{CFT}.   For dimensions higher than one, $n\geq 2$, we conjecture  that this is also the case. That is,  in parallel with the results for the standard semi-linear Schr\"odinger equation, we make the following conjecture. 
\begin{Conjecture}
For $s>\frac n2,$ $u_0\in D(\cl_c)$ there exists $T>0$ such that \eqref{20} is locally well-posed and \eqref{conserve} are conserved up to an eventual blow-up time. In addition,  if  $0<\sigma<\frac{2s}{n} -1$,  the solutions are  global, whereas for $\sigma\geq \f{2s}{n}-1$, finite time blow-up is possible, for some initial data. 
\end{Conjecture}
We are now ready to state our orbital stability results. 
\begin{theorem}
	\label{theo:20} 
	Let  $n\geq 1$, $\om>0$, $s>\f{n}{2}$, $0<\si<\f{2s}{n}-1$. In addition, assume     continuous dependence on initial data and globality of the solutions close to $\phi_\om$, as outlined above.  Then, the solitons $e^{i \om t} \phi_\om$ is orbitally stable. 
\end{theorem}
We plan our paper as follows. In Section \ref{sec:2}, we prove the Pohozaev's identities, which in turn imply the necessary conditions for existence of the waves. Then, we discuss a  self-adjoint realization of the operators $(-\De)^s+\la- c \de_0$ for $\la>0, c>0$.  

In Section \ref{sec:3}, we first provide a variational construction of the waves $\phi_\om$. 
The special relation to the Sobolev embedding $H^s(\rn)\hookrightarrow L^\infty(\rn), s>\f n2$ is highlighted. The precise results are stated in the explicit formulas in Proposition \ref{r:4}. Finally, in Section \ref{sec:3.4}, we discuss the lower part of the spectrum for operators in the form $(-\De)^s+\la - \mu\de_0$. In the particular case of the linearized operator  $\cl_+$, this yields the non-degeneracy of the waves, which in this case takes the form $Ker(\cl_+)=\{0\}$, due to the broken translational symmetry. 

In Section \ref{sec:4}, we start with  a short introduction to the instability index count theory in general, and then we apply it to the spectral stability of the waves $\phi_\om$. We explicitly calculate the relevant Vakhitov-Kolokolov quantity $\dpr{\cl_+^{-1} \phi_\om}{\phi_\om}$, which provides the stability characterization of the waves described in Theorem \ref{theo:12}. Finally, under the necessary and sufficient condition for spectral stability,  $\dpr{\cl_+^{-1} \phi_\om}{\phi_\om}<0$, we   derive the  coercivity of $\cl_+$ on $\{\phi_\om\}^\perp$, which is of course crucial in the proof of the orbital stability.

\section{Preliminaries} 
\label{sec:2}
We use the standard notations for the $L^p, 1\leq p\leq \infty$ spaces. 
The Sobolev norms  $\|\cdot\|_{W^{s,p}}$ are given by
$$
\|f\|_{\dot{W}^{s,p}}=\|(-\De)^{\f{s}{2}} f\|_{L^p}; \ \ \|f\|_{W^{s,p}}=\|(-\De)^{\f{s}{2}} f\|_{L^p}+\|f\|_{L^p}, 1<p<\infty,
$$
while the corresponding spaces are the completions of Schwartz functions 
$\cs$ in these norms. 

Of particular importance will be the Sobolev embedding, $\dot{W}^{s,p}(\rn)\hookrightarrow  L^q(\rn)$, for $1<p<q<\infty: s\geq n\left(\f{1}{p}-\f{1}{q}\right)$. Also,  recall that for $s>\f{n}{p}$, there is the embedding\footnote{Here $\{x\}=x-[x]$, where  $[x]=\max\{n: n\leq x\}$}  $W^{s,p}\hookrightarrow 
C^{[s-\f np],\ga}(\rn): 0<\ga<s-\f{n}{p}$. As is well-known, the embedding  $H^{\f n2}(\rn)\hookrightarrow L^\infty(\rn)$ fails, but sometimes and   useful replacement estimate   is the following for all $\de\in (0,\f n2)$, 
\begin{equation}
\label{017}
\|f\|_{L^\infty} \leq C_\de (\|f\|_{\dot{H}^{\f n2 - \de}}+ \|f\|_{\dot{H}^{\f n2 +\de}}). 
\end{equation}
\subsection{Pohozaev's identities and consequences}
\label{sec:2.2} 
We would like to address the question for existence of solutions for the profile equation \eqref{30}. Eventually, we will write them down explicitly, but first, we need to identify some necessary conditions on the parameters, which turn out to be essentially sufficient as well. The approach here is classical (even though our problem is certainly not) - we build some Pohozaev's identities, which proceeds  by establishing relations between various norms of the eventual solution $\phi$, which are {\it a priori} assumed finite.  As a consequence, we find that the parameters must meet certain constraints. 
\begin{proposition}
	\label{prop:10} 
	Let $\phi\in H^s(\rn)\cap C(\rn)$  be a weak solution of \eqref{30}. Then, 
	\begin{eqnarray}
	\label{40} 
& & 	\|\phi\|_{L^2}^2=\f{2s-n}{2s\om} |\phi(0)|^{2\si+2} \\
\label{42} 
& & \|(-\De)^{\f{s}{2}} \phi\|_{L^2}^2= \f{n}{2s} 	|\phi(0)|^{2\si+2}.
	\end{eqnarray}
\end{proposition}
\begin{proof}
	Testing \eqref{30} with $\phi$ itself results in 
	\begin{equation}
	\label{45} 
		\|(-\De)^{\f{s}{2}} \phi\|_{L^2}^2 + \om \|\phi\|_{L^2}^2 - |\phi(0)|^{2\si+2}=0.
	\end{equation}
Next, we test \eqref{30} against $x\cdot \nabla\Psi$, for a test function $\Psi$. We obtain, by taking into account the commutation relation $[(-\De)^s, x\cdot \nabla]=2s(-\De)^s$, 
	\begin{eqnarray*}
	\dpr{(-\De)^{\f{s}{2}} \phi}{(-\De)^{\f{s}{2}} [x\cdot \nabla\Psi]} &=& \dpr{\phi}{x\cdot \nabla (-\De)^s \Psi}+2s \dpr{(-\De)^{\f{s}{2}} \phi}{(-\De)^{\f{s}{2}} \Psi}=\\
	&=& -\dpr{x\cdot \nabla \phi}{(-\De)^s \Psi}+(2s-n) \dpr{(-\De)^{\f{s}{2}} \phi}{(-\De)^{\f{s}{2}} \Psi}=\\
	&=& -\dpr{(-\De)^{\f{s}{2}}[x\cdot \nabla \phi]}{(-\De)^{\f{s}{2}} \Psi}+(2s-n) \dpr{(-\De)^{\f{s}{2}} \phi}{(-\De)^{\f{s}{2}} \Psi}.
	\end{eqnarray*}
	This implies 
	$$
	\dpr{(-\De)^{\f{s}{2}} \phi}{(-\De)^{\f{s}{2}} [x\cdot \nabla\Psi]}+\dpr{(-\De)^{\f{s}{2}}[x\cdot \nabla \phi]}{(-\De)^{\f{s}{2}} \Psi}=(2s-n) \dpr{(-\De)^{\f{s}{2}} \phi}{(-\De)^{\f{s}{2}} \Psi}.
	$$
	Note that the right-hand side of this expression makes sense for\footnote{one can formally take limits of $\Psi_n: \|\Psi_n-\phi\|_{H^s}\to 0$},  $\Psi=\phi$ whence 
	\begin{equation}
	\label{55} 
	\dpr{(-\De)^{\f{s}{2}} \phi}{(-\De)^{\f{s}{2}} [x\cdot \nabla\Psi]}=(s-\f{n}{2}) 
	\|(-\De)^{\f{s}{2}} \phi\|^2.
	\end{equation}
Also\footnote{note that $\phi\in H^1(\rn)$ makes this well-defined}
$$
 \dpr{\phi}{x\cdot \nabla\Psi}=-n\dpr{\phi}{\Psi} - \dpr{x\cdot \nabla \phi}{\Psi},
$$
which also makes sense for $\Psi=\phi$, whence 
\begin{equation}
\label{57} 
\dpr{\phi}{x\cdot \nabla\Psi}=-\f{n}{2} \|\phi\|^2.
\end{equation}
Finally, we claim that $\dpr{\de_0}{x\cdot \nabla \Psi}=0$ for each test function $\Psi$. Indeed, Introduce a radial function $V:\rn\to \rone$, which is smooth and non-negative function, supported on $\bb:=\{x\in \rn: \|x\|<1\}$ and normalized so that $\int_{\rn} V(x) dx=1$. It is well-known, that in a distribution sense, one can approximate $N^n V(N x)\to \de_0$.  That is $\lim_{N\to \infty} 
\dpr{ N^n V(N \cdot)}{f}=f(0)$. So, 
	\begin{eqnarray*}
\dpr{\de_0}{x\cdot \nabla \Psi} &=& \lim_{N\to \infty} N^n \sum_{j=1}^n \int_{\rn} V(N x) x_j\p_j \Psi(x) dx=\\
&=& 
\lim_{N\to \infty}\left[-n N^n \int_{\rn} V(N x)  \Psi(x) dx  - N^{n+1} \int_{\rn} |x| V'(Nx) \Psi(x)dx \right]=0,
\end{eqnarray*}
	since 
	$$
	 N^{n+1} \int_{\rn} |x| V'(Nx)dx=\int_{\rn} |y| V'(y) dy= |{\mathbb S^{n-1}}|
	\int_0^\infty V'(\rho) \rho^n d\rho=-n\int_0^\infty V(\rho) \rho^{n-1} d\rho=-n.
	$$
	Putting  $\dpr{\de_0}{x\cdot \nabla \Psi}=0$ together with \eqref{55}, \eqref{57}, implies 
	\begin{equation}
	\label{50}
	(s-\f{n}{2}) 
	\|(-\De)^{\f{s}{2}} \phi\|^2 - \f{\om n}{2}  \|\phi\|^2=0.
	\end{equation}
	Solving the system of equations \eqref{45} and \eqref{50} results in the relations \eqref{40} and \eqref{42}.
\end{proof}
An immediate corollary of these results is, using the positivity of the various norms in \eqref{40}, \eqref{42} is given by Proposition \ref{r:2}. Namely, either $\om>0, s>\f n2$ or $\om<0, s< \f n2$. 
Clearly, the case $\om>0, s>\f{n}{2}$ is a more physical situation to consider - after all, one has the embedding $H^s(\rn) \hookrightarrow C(\rn)$ and hence functions in the class $H^s(\rn)$ are automatically continuous. 
\subsection{The self-adjoint operators $(-\De)^s+\la - c\de_0$}
\label{sec:2.3}
 In this section, we introduce the necessary self-adjoint extensions of the operators formally introduced as $(-\De)^s+\la - c\de_0$. There has been quite a bit of recent work on the subject, see \cite{ACCT, ADFT, AN1, ANO2, CFT} among others. In these papers, the authors  have introduced various (and sometimes all) self-adjoint extensions of such objects, under different assumptions on the parameters. As dictated by the results of Proposition  \ref{r:2},  we work under  the assumption $s>\f{n}{2}$, which simplifies matters quite a bit, in the sense that the self-adjoint extension, which generates the standard quadratic form,  is canonical. 
 
 More specifically, for given constants $\la>0, c>0$,  we introduce the skew-symmetric quadratic form 
 $$
 \cq_c(f,g)= \dpr{\sqrt{((-\De)^s+\la) f}}{\sqrt{((-\De)^s+\la) g}}  - c f(0) \bar{g}(0), f,g\in D(\cq)
 $$
 with domain $D(\cq)=H^s(\rn)$. Note that as $D(\cq)\subset C(\rn)$, the values $f(0), g(0)$ make sense. In addition, the form $\cq$ is bounded from below. This is  a consequence of the Sobolev embedding $H^\al\hookrightarrow L^\infty(\rn), \al>\f{n}{2}$. 
 Indeed, choose  $\al: \f{n}{2}<\al<s$ and estimate via the Sobolev and the Gagliardo-Nirenberg's inequalities 
 \begin{eqnarray*}
 \cq_c(f,f)\geq c_\la \|f\|_{H^s}^2 - k_\al \|f\|_{H^\al}^2 \geq c_\la \|f\|_{H^s}^2 - 
 k_{\al}( \f{c_\la}{2 k_{\al}} \|f\|_{H^s}^2 + d_{\al,\la} \|f\|_{L^2}^2)\geq D_{\al,\la} \|f\|_{\dot{H}^s}^2 - M_{\al, \la} \|f\|_{L^2}^2.
 \end{eqnarray*}
 In addition, $\cq$ is closed form, as $\|f\|_{H^s}^2\sim \cq(f,f)+M \|f\|^2$, for large enough $M$. 
 According to the standard theory for quadratic forms, see Theorem VIII.15 in \cite{RS},  there is an unique self-adjoint operator $\cl_c$, so that 
$$
 D(\cl_c)\subset D(\cq), \ \ \cd_c(f,g)=\dpr{\cl_c f}{g}, \ \ \forall f\in D(\cl_a), g\in D(\cq).
 $$
 Identifying the exact form of $\cl_c$ may not be an easy task, in general.  In our case, this is not so hard, as the operator  has been essentially constructed in previous works, see \cite{CFT} for the one dimensional case. We follow their notations and approach. 
 To this end, introduce  the Green's function of the operator $(-\De)^s+\la$, namely the  function $\cg_s^\la$, so that 
 $$
 ((-\De)^s+\la) \cg_s^\la=\de_0.
 $$
 By taking the Fourier transform, we can write the following formula for $\cg_s^\la$
 $$
 \widehat{\cg_s^\la}(\xi)=\f{1}{(2\pi |\xi|)^{2s}+\la}.
 $$
 Clearly, since $s>\f{n}{2}$, $\cg_s^\la\in H^s(\rn)\subset C(\rn)$. Introduce the domain of the operator $\cl_c$ as
 \begin{equation}
 \label{65}
 D(\cl_c)=\{\psi\in H^s(\rn): \psi=g+c \psi(0) \cg_s^\la, g\in H^{2s}(\rn)\}\subset H^s(\rn).
 \end{equation}
 With this domain,  its action is  defined as 
 \begin{equation}
 \label{70}
  \cl_c \psi :=((-\De)^s+\la) g.
 \end{equation}
Note that for $\psi\in D(\cl_c)$ and $h\in H^s(\rn)=D(\cq)$, we have 
\begin{eqnarray*}
\dpr{\cl_c \psi}{h} &=& \dpr{((-\De)^s+\la) g}{h}=\dpr{\sqrt{(-\De)^s+\la}\psi}{\sqrt{(-\De)^s+\la} h}-c\psi(0) \dpr{((-\De)^s+\la)\cg_s^\la}{h}\\
&=& \dpr{\sqrt{(-\De)^s+\la}\psi}{\sqrt{(-\De)^s+\la} h} - c\psi(0)\bar{h}(0)=\cq_c(\psi,h).
\end{eqnarray*}
Thus, $\cl_c$ is a closed symmetric operator, with a quadratic form precisely $\cq$. Note that the role of the constant $\la$ in the definition is to ensure that the function $\widehat{\cg_s^\la}$ has no singularity  at $\xi=0$. Also, for every $\tilde{\la}>0$, we have $\cl_c^{\tilde{\la}}=\cl_c^\la+\tilde{\la}-\la$. 

We now need to show that $\cl_c$ is precisely the unique self-adjoint operator with this property. 
\begin{lemma}
	\label{le:12} 
	The closed symmetric operator $\cl_c$, with domain given in \eqref{65} and whose action is defined in \eqref{70},  is self-adjoint. 
\end{lemma}
 \begin{proof}
 	For technical reasons, let us first assume the condition 
 	\begin{equation}
 	\label{cond} 
 	c\cg_s^\la(0)\neq 1.
 	\end{equation}
 	With that,  we work on a different representation on $D(\cl_c)$. More precisely, we would like to write $\psi$ purely  in terms of $g$. To this end, we evaluate the identity relating $\psi$ and $g$ at $x=0$. We obtain the equation for $\psi(0)$ 
 	$$
 	\psi(0)=g(0)+c \psi(0) \cg_s^\la(0).
 	$$
 	This equation has a solution, under the condition \eqref{cond}, 
 	\begin{equation}
 	\label{80}
 		\psi(0)=\f{g(0)}{1-c \cg_s^\la(0)}.
 	\end{equation}
 	One can now write, for $c\neq \f{1}{\cg_s^\la(0)}$, 
 	$$
 	D(\cl_c)=\{\psi\in L^2(\rn): \psi=g +c  \cg_s^\la \f{g(0)}{1-c \cg_s^\la(0)}, g \in H^{2s}(\rn)\},
 	$$
 	which describes $D(\cl_c)$ purely in terms of an arbitrary function $g\in H^{2s}(\rn)$. 
 	
 	In order to show that $\cl_c=\cl_c^*$, it suffices to show that it has a real number in its resolvent set, see Corollary on p. 137, \cite{RS2}. To this end, let $M>>1$, and we will show that $-M-\la\in \rho(\cl_c)$. Let $f\in L^2(\rn)$ is arbitrary and consider 
 	\begin{equation}
 	\label{75}
 		(\cl_c+M-\la) \psi=f.
 	\end{equation}
 	This is of course equivalent to the equation $((-\De)^s+M)g=f$, where 
 	$$
 	\psi=g +c  \cg_s^\la \f{g(0)}{1-c \cg_s^\la(0)}. 
 	$$
 	which has the unique solution 
 	$
 	g=((-\De)^s+M)^{-1} f\in H^{2s}(\rn).
 	$
 	Thus, we can uniquely  solve \eqref{75} as follows 
 	$$
 	\psi=g +c  \cg_s^\la \f{g(0)}{1-c \cg_s^\la(0)}, \ \ g=((-\De)^s+M)^{-1} f\in H^{2s}(\rn).
 	$$
 	In terms of estimates $\|g\|_{H^{2s}}\leq C_M \|f\|_{L^2}$ and consequently 
 	$$
 	\|\psi\|_{L^2}\leq \|g\|_{L^2}+ C|g(0)|\leq  \|g\|_{H^s}\leq C_M \|f\|_{L^2}.
 	$$
 	This shows that all $\cl_c$, with $c$ satisfying \eqref{cond} are self-adjoint. What about $c$, which fails \eqref{cond}? In this case 
 	$$
 1=	c \cg_s^\la(0)= c \int_{\rn} \f{1}{(2\pi |\xi|)^{2s}+\la} d\xi
 	$$
 	It follows that for every $\tilde{\la}\neq \la$, say $\tilde{\la}>\la$, we have that $c \cg_s^{\tilde{\la}}(0)\neq 1$. Thus, following the scheme described in the previous arguments, the operator $\cl_c^{\tilde{\la}}$, formally defined through $(-\De)^s+\tilde{\la}- c \de_0$ is self-adjoint. This means that 
 	$$
 	\cl_c=\cl_c^\la=\cl_c^{\tilde{\la}}+(\la-\tilde{\la})Id,
 	$$
 	is self-adjoint as well. 
 \end{proof}
 \begin{rmk}
 	In particular, we have the following important formula for the action of $\cq_c$ on functions $\psi\in H^s$, with $\psi(0)=0$, 
 	\begin{equation}
 	\label{112} 
 	\cq_c(\psi,\psi)=\|(-\De)^{\f{s}{2}}\psi\|_{L^2}^2+\la \|\psi\|_{L^2}^2.
 	\end{equation}
 \end{rmk}

 \section{Variational construction of the waves $\phi_\om$ and spectral consequences} 
 \label{sec:3}
 We first construct,  in a variational manner,  some approximate solutions to the elliptic profile problem \eqref{30}. This will turn out to be important in our subsequent considerations. 
 \subsection{Variational constructions} 
Let $\om, \si>0$. For a radial function $V:\rn\to \rone$ as before\footnote{i.e. $V$ is non-negative, radial, smooth and supported on the unit ball $\bb\subset \rn$, with $\int_{\bb} V(x) dx=1$} and $N>>1$, consider the functional 
$$
I_{\omega,N}[u]=\frac{\int_{\mathbb{R}^n}|(-\Delta)^{s/2}u|^2dx+\omega\int_{\mathbb{R}^n}u^2 dx}{\left(\int_{\mathbb{R}^n} N^nV(Nx)|u|^{2\sigma+2}dx\right)^{\frac{1}{\sigma+1}}}.
$$
and the corresponding  unconstrained variational problem
$
I_{\om,N}[u] \to \min.
$
Clearly, $I_{\om,N}[u]>0$, so  its optimal value is well-defined 
$$
m_N(\omega):=\inf_{u\in\mathcal{S}, u\neq 0}I_{\omega,N}[u].
$$
 \begin{proposition}
 	\label{prop:22}
 	Let $s>\frac{n}{2}$.  
 	Then the unconstrained minimization problem 
 	\begin{equation}
 	\label{35}
 	I_{\om,N}[u]\to \min
 	\end{equation}
 	has a real-valued solution $\phi_N\in H^s(\rn)\cap L^{\infty}$, in particular $m_N(\om)>0$. Moreover, $\phi_N$ may be chosen to satisfy 
 	$$
 	N^n \int_{\rn} V(Nx)|\phi_N(x)|^{2\sigma+2}dx=1. 
 	$$
 	Finally, $\phi_N$ satisfies the Euler-Lagrange equation
 	\begin{equation}
 	\label{400}
 	(-\Delta)^s \phi_N +\om  \phi_N- m_N(\om)  N^nV(Nx)|\phi_N|^{2\sigma}\phi_N=0
 	\end{equation}
 	in distributional sense.
 \end{proposition}
 
 \begin{proof}
 	Since $\|V\|_{L^1}=1$, we have  for $u\in H^s(\rn)\subset L^{\infty}$,   
 	\begin{equation}
 	\label{160} 
 	\left(N^n \int_{\rn} V(Nx)|u(x)|^{2\sigma+2}dx\right)^{\frac{1}{\sigma+1}}\leq \|u\|_{L^\infty(\rn)}^2\leq C \|u\|_{H^s(\rn)}^2,
 	\end{equation}
 	whence \eqref{35}  is a well-posed variational problem and $m_N(\om)>0$. Next, 
 	due to dilation properties of the functional $I_{\om,N}$,  we can assume that  the infimum is taken only over functions with the normalization property 
 	$$
 N^n	\int_{\mathbb{R}^n}   V(Nx)|u(x)|^{2\sigma+2}dx=1.
 	$$
 	 Let $u_k$ be a minimizing sequence such that $\int_{\mathbb{R}^n}N^nV(Nx)|u_k|^{2\sigma+2}dx=1$ and hence 
 	 $$
 	 \lim_{k}(\|(-\Delta)^{\f{s}{2}}u_k\|_{L^2}^2+\omega \|u_k\|_{L^2}^2)=m_N(\omega).
 	 $$ 
 	By weak compactness, we can select a weakly convergent subsequence (which we assume is just $\{u_k\}$), $u_k\rightharpoonup u$. By the lower semi-continuity of the norms, with respect to weak convergence 
 	\begin{equation}
 	\label{115} 
 	\|(-\Delta)^{\f{s}{2}}u\|_{L^2}^2+\omega \|u\|_{L^2}^2\leq \liminf_k (\|(-\Delta)^{\f{s}{2}}u_k\|_{L^2}^2+\omega \|u_k\|_{L^2}^2)=m_N(\omega).
 	\end{equation}
 	 We now show that $\{u_k\}$ is pre-compact in $C(\bb)$. Indeed, since $s>\frac{n}{2}$, we have  by the Sobolev embedding that  
 		\begin{equation}
 		\label{117} 
 	\|u_k\|_{C^{\gamma}(\rn)}\leq C \|u_k\|_{H^s},
 		\end{equation}
 	 for $0<\gamma<\{s-\frac{n}{2}\}$. Consequently, $u_k$ are uniformly H\"older-continuous, hence equicontinuous as elements of $C(\bb)$.  Also,  $\{u_k\}$ is a totally bounded by \eqref{117}. By Arzela-Ascolli, we have that $\{u_k\}_{k=1}^{\infty}$ is pre-compact in $C[0,1]$ , i.e for a subsequence, which we again assume it is just $u_k$, we have that $u_{k}\rightrightarrows_{\bb} u$. It is now clear that 
 	 \begin{equation}
 	 \label{180} 
 	 1=\lim_k N^n \int_{\rn}V(Nx)|u_k(x)|^{2\sigma+2}dx=
 	 N^n \int_{\rn} V(Nx)|u(x)|^{2\sigma+2}dx.
 	 \end{equation}
 	 Thus, by \eqref{115}and \eqref{180}, we conclude that 
 	 $
 	 I_{\om,N}[u]\leq m_N(\om). 
 	 $
 	This,  by the definition of $m_N(\om)$ means that $I_{\om,N}[u]=m_N(\om)$. In particular, 
 	$$
 	\|(-\Delta)^{\f{s}{2}}u\|_{L^2}^2+\omega \|u\|_{L^2}^2=m_N(\omega),
 	$$
 	so $u$ actually solves the minimization problem \eqref{35}. This is the solution $\phi_N$ that we were interested in. 
 	
 	Next we show that the minimizer satisfies the Euler Lagrange equation. To that end take an arbitrary test function $h$ and let $\epsilon>0$ consider $u=\phi_N+\epsilon h,$ and recall that $\int N^nV(Nx)|\phi_N|^{2\sigma+2}dx=1.$ Since $\phi_N$ is a minimizer we have that 
 	$
 	I_{\om,N}[u]\geq m_N(\om). 
 	$
 	Expanding in powers of $\eps$, we obtain 
 	$$
 	\int |(-\Delta)^{s/2}(\phi_N+\epsilon h)|^2dx+\om\int (\phi_N+\epsilon h)dx
 	=m_N(\om)+2 \epsilon\langle((-\Delta)^s+\om)\phi_N,h\rangle+O(\epsilon^2).
 	$$
 	Similarly, 
 	\begin{align*}
 	\int V(Nx)|\phi_N+\epsilon h|^{2\sigma+2}dx
 	&= \int V(Nx)|\phi_N|^{2\sigma+2}dx +(2\si+2) \epsilon \int V(Nx)|\phi_N|^{2\sigma}\phi_N h  +O(\epsilon^2)\\
 	&= 1+(2\sigma+2) \epsilon \int V(Nx)|\phi_N|^{2\sigma}\phi_N h+O(\epsilon^2)
 	\end{align*}
 	Thus, 
 	\begin{align*}
 	I_{\om,N}&=\frac{m_N(\om)+2\epsilon\langle((-\Delta)^s+\om)\phi_N,h\rangle+O(\epsilon^2)}{
 		1+2 \epsilon \int N^nV(Nx)|\phi_N|^{2\sigma}\phi_N h dx +O(\epsilon^2)}\\
 	&=m_N(\om)+ 2 \epsilon\langle((-\Delta)^s+\om)\phi_N-m_N(\om)N^nV(Nx)|\phi_N|^{2\sigma}\phi_N,h\rangle+O(\epsilon^2)
 	\end{align*}
 	Since this hold for any arbitrary test function $h$ and any 
 	$\epsilon>0$ we have that $\phi_N$ solves \eqref{400}. 
 \end{proof}
 Next, we have the following technical result. 
 \begin{lemma} 
 	\label{le:2}
 	There exists constants $C_1(\om), C_2(\om)$, but independent on $N$, so that 
 	$$
 C_1(\om)\leq m_N(\om)\leq C_2(\om).
 	$$
 	 Furthermore, the sequence $\{\phi_N\}_{N=1}^{\infty}$, is a  pre-compact in every set of the form $C(K)$, where $K$ is a compact subset of $\rn$. 
 	
 \end{lemma}
 \begin{proof}
 	The lower bound, with a constant independent on $N$ follows from \eqref{160}. The upper bound follows by testing against a concrete  function like $u_0(x)=e^{-|x|^2}$. Since $\f{1}{3}<u_0(x)\leq 1$, on the support of $V(N x), N\geq 1$, we have that 
 	$$
 m_N(\om)\leq 	I_{\om,N}[u_0]\leq 9 \left(\|(-\De)^{\f{s}{2}} u_0\|_{L^2}^2+ \om\| u_0\|_{L^2}^2\right)=:C_2(\om).
 	$$
 	Next, since $\phi_N$ satisfy $N^n \int_{\rn} V(N x) |\phi_N(x)|^{2\si+2} dx=1$, we have that $I_{\om,N}[\phi_N]=\|(-\De)^{\f{s}{2}} \phi_N\|_{L^2}^2+ \|\phi_N\|_{L^2}^2=m_N(\om)$. Thus, by Sobolev embedding 
 	$$
 	\|\phi_N\|_{C^\ga(\rn)}\leq C \|\phi_N\|_{H^s}\leq C(\om) m_N(\om)\leq C_3(\om).
 	$$
 	for $0<\ga<\min\{1, s-\f{n}{2}\}$. It follows that for  each compact $K\subset \rn$, $\{\phi_N\}$ is pre-compact in $C(K)$ by Arzela-Ascolli's theorem. 
 \end{proof}
 Clearly, Lemma \ref{le:2} allows us to take convergent (sub) sequence as $N\to \infty$. We wish to learn what the limit is expected to be. It turns out that it is nothing but the minimizer for the Sobolev inequality $H^s(\rn)\hookrightarrow L^\infty(\rn)$. We justify  that in the next section. 
 \subsection{Relation to the minimizers for the Sobolev embedding  $H^s(\rn)\hookrightarrow L^\infty(\rn)$}
 For $s>\f{n}{2}, \om>0$, we study  up the  functional
 $$
 J_\om[u]=\f{\|(-\De)^{\f{s}{2}} u\|_{L^2}^2+ \om \|u\|_{L^2}^2}{\|u\|_{L^\infty}^2}
 $$
 and the corresponding minimization problem $J_\om[u]\to \min$. Finally, denote 
 $$
 c^2(\om):=\inf_{u\in \cs: u\neq 0} J_\om[u].
 $$
 The described optimization problem has a clear analytical interpretation, namely that $c$ is the exact constant in the Sobolev embedding estimate 
 $$
 c(\om) \|u\|_{L^\infty} \leq \vertiii{u}_{H^s}:=\sqrt{\|(-\De)^{\f{s}{2}} u\|_{L^2}^2+ \om \|u\|_{L^2}^2}.
 $$
 We now from the Sobolev embedding $H^s(\rn)\hookrightarrow L^\infty(\rn)$ that $c$ is well-defined and we can alternatively introduce it as follows $c(\om) = \sup\{C>0: C \|u\|_{L^\infty}\leq \vertiii{u}_{H^s}, \forall u\in \cs\}$. 
 
 Another useful observation is that one can assume, without loss of generality, that in the infimum procedure described above, $\|u\|_{L^\infty}$ is replaced by $|u(0)|$. That is 
 $$
 c^2(\om)=\inf_{u\in H^s: u(0)\neq 0} \f{\|(-\De)^{\f{s}{2}} u\|_{L^2}^2+ \om \|u\|_{L^2}^2}{|u(0)|^2}.
 $$
 \begin{lemma}
 	\label{le:20} 
 	Let $s>\f{n}{2}, \om>0$ and $\ga<\min(1, s-\f{n}{2})$. Then, there exists $C=C(s,\om, \ga)$, so that 
 	\begin{equation}
 	\label{200} 
 	c^2(\om)\leq m_N(\om)\leq c^2(\om)+ C N^{-\ga}
 	\end{equation}
 \end{lemma}
 \begin{proof}
 	By \eqref{160}, we see that for every $N\geq 1$, $I_{\om, N}\geq J_\om$, whence $m_N(\om)\geq c^2(\om)$. 
 	
 	For the opposite inequality, observe first that since $m_N(\om)\leq C_2(\om)$, we can take 
 	$$
 	m_N(\om)=\inf_{u\in \cs: u\neq 0} I_{\om, N}[u]=\inf\limits_{ N^n \int_{\rn} V(N x) |\phi_N(x)|^{2\si+2} dx=1; \ \ \vertiii{u}_{H^s} \leq 10 C_2} I_{\om, N}[u].
 	$$
 	So, let $u\in H^s:  N^n \int_{\rn} V(N x) |u(x)|^{2\si+2} dx=1; \ \ \vertiii{u}_{H^s} \leq 10 C_2$. Recall that for every $q>1$, there is $C_q$, so that for $a>0, b>0$
 	$
 	|a^q-b^q|\leq C_q|a-b| (a^{q-1}+b^{q-1}).
 	$
 	As a consequence, and by Sobolev embedding
 	$$
 	\left||u(x)|^{2\si+2}-|u(0)|^{2\si+2}\right|\leq C_\si |u(x)-u(0)| \|u\|_{L^\infty}^{q-1}\leq 
 	C_{\ga, \si} |x|^\ga \|u\|_{C^\ga(\rn)}^q\leq 	C_{\ga, \si} |x|^\ga \|u\|_{H^s}^q
 $$
 and since $\vertiii{u}_{H^s} \leq 10 C_2$, we conclude 
 	\begin{equation}
 	\label{210} 
 	\left||u(x)|^{2\si+2}-|u(0)|^{2\si+2}\right|\leq C_{\ga, \si, \om} |x|^\ga. 
 	\end{equation}
 	It follows that 
 	\begin{eqnarray*}
 \left||u(0)|^{2\si+2}-1\right|	&=&  \left||u(0)^{2\si+2}-N^n \int_{\rn} V(N x) |u(x)|^{2\si+2} dx\right| = \\
 &=& 
 N^n\left|\int_{\rn} V(N x) [|u(x)|^{2\si+2}-|u(0)|^{2\si+2} dx\right|\leq C_{\ga, \si, \om} N^n \int_{\rn} V(N x) |x|^\ga dx\\
 &\leq& C_{\ga, \si, \om}  N^{-\ga} \int_{\rn} V(y) |y|^\ga dy\leq C_{\ga, \si, \om}  N^{-\ga},
 	\end{eqnarray*}
 	so $|u(0)|\leq 1 + C_{\ga, \si, \om}  N^{-\ga}$. 
 	It follows that 
  	\begin{eqnarray*}
  	m_N(\om) &=& \inf\limits_{ N^n \int_{\rn} V(N x) |\phi_N(x)|^{2\si+2} dx=1; \ \ \vertiii{u}_{H^s} \leq 10 C_2} 
  	\|(-\De)^{\f{s}{2}} u\|_{L^2}^2+ \om \|u\|_{L^2}^2 \leq \\
  	&\leq & 	(1+C_{\ga, \si, \om}  N^{-\ga}) \inf\limits_{ \vertiii{u}_{H^s} 
  		\leq   10 C_2, u(0)\neq 0} \f{\|(-\De)^{\f{s}{2}} u\|_{L^2}^2+ \om \|u\|_{L^2}^2}{|u(0)|^2} \leq 
  	c^2+ C_{\ga, \si, \om}  N^{-\ga}.
  	\end{eqnarray*}

 \end{proof}
 We now take need to take limit as $N\to \infty$. In view of our discussion so far, it is not surprising that  this yields  the minimizers for the Sobolev embedding $H^s(\rn)\hookrightarrow L^\infty(\rn)$. In turn, this allows us to present an explicit formula for the solutions of \eqref{30} and to interpret them as minimizers of the Sobolev embedding problem. 
 \subsection{Description of  the solutions for the profile equation \eqref{30}}
 \begin{lemma}
 	\label{le:40}
 	Let $s>\f{n}{2}, \om>0$. Then, for every constant $C\neq 0$, the function 
 	\begin{equation}
 	\label{220} 
 		\hat{\phi}(\xi)=\f{C}{(2\pi|\xi|)^{2s}+\om},
 	\end{equation}
 	is a minimizer of the problem $\min_{u\in H^s} J_\om[u]$. In particular, the optimal Sobolev constant is given by the formula 
 	$$
 	c^2(\om)=\left(\int_{\rn} \f{1}{(2\pi|\xi|)^{2s}+\om} d\xi\right)^{-1}. 
 	$$
 \end{lemma}
 \begin{proof}
 	From Lemma \ref{le:20}, it follows that $\lim_N m_N(\om)=c^2(\om)$. In addition, as we have pointed out, maximizers can be taken, with the property $\|\phi_N\|_{H^s}\leq C(\om)$. As $H^s(\rn)$ embeds in $C^\ga(\rn), 0<\ga<s-\f{n}{2}$ and this is compact embedding on bounded domains, we can select 
 $$
 \phi_N: N^n \int_{\rn} V(N x) |\phi_N(x)|^{2\si+2} dx=1, 
 $$
  so that 
 $\phi_N$ is  uniformly convergent, on the compact subsets of $\rn$ 
 to $\phi\in H^s(\rn)$. 
 
 We will show that $\phi(0)=1$ and $\phi$ is in the form \eqref{220}. We have, for each $N\geq 1$, 
 	\begin{eqnarray*}
 		\left|1-|\phi(0)|^{2\si+2}\right| &\leq & N^n\int_{\rn} V(N x)\left||\phi_N(x)|^{2\si+2}-|\phi(0)|^{2\si+2}\right| dx\\
 		&\leq & C_\si (\|\phi_N\|_{L^\infty}^{2\si+1}+|\phi(0)|^{2\si+1}) 
 		N^n\int_{\rn} V(N x) |\phi_N(x)-\phi(0)|dx. 
 	\end{eqnarray*}
 	But $\|\phi_N\|_{L^\infty}\leq \|\phi_N\|_{H^s}<C(\om)$, while 
 	$$
 	|\phi_N(x)-\phi(0)|\leq |\phi_N(x)-\phi_N(0)|+|\phi_N(0)-\phi(0)|\leq C_\ga |x|^\ga + |\phi_N(0)-\phi(0)|. 
 	$$
 	Plugging this back in our estimate for $|1-|\phi(0)|^{2\si+2}|$, we obtain, for each $0<\ga<s-\f{n}{2}$, 
 	$$
 	|1-|\phi(0)|^{2\si+2}|\leq  C |\phi_N(0)-\phi(0)|+ C N^n\int_{\rn} V(N x)  |x|^\ga dx\leq C |\phi_N(0)-\phi(0)| + C N^{-\ga}.
 	$$
 	Clearly, the expression on the right goes to zero as $N\to \infty$, as 
 	$\phi_N\rightrightarrows_{\bb} \phi$. By adjusting the sign of $\phi_N$, if necessary, this implies that we can take  $\phi(0)=\lim_N \phi_N(0)=1$.
 	
 	Next, $\phi_N$ satisfies the Euler-Lagrange equation \eqref{400}. Test this equation with $\psi$. We obtain 
 	\begin{equation}
 	\label{230} 
 	\dpr{\phi_N}{((-\De)^s+\om)\psi}=m_N(\om) N^n \int_{\rn} V(N x) |\phi_N|^{2\si} \phi_N(x) \psi(x) dx.
 	\end{equation}
 	Taking limits in $N$ then yields, after taking into account $\phi(0)=1$, 
 	\begin{equation}
 	\label{235} 
 	\dpr{\phi}{((-\De)^s+\om)\psi}=c^2(\om) \psi(0).
 	\end{equation}
 In other words, $\phi$ satisfies the equation 
 \begin{equation}
 \label{240}
 ((-\De)^s+\om)\phi-c^2 \de_0=0.
 \end{equation}
 	in a distributional sense. 
 	
 	By taking   $\psi$ in \eqref{230}, to be an appropriate approximation of the function $\cg^\om_s(\cdot+x)$, we conclude that 
 	$$
 	\phi(x) = const. \cg^\om_s(x)
 	$$
 	which is of course the same as \eqref{220}. Additionally, by testing \eqref{240} by $\phi$ itself, we obtain 
 	$$
 	\|(-\De)^{\f{s}{2}} \phi\|_{L^2}^2 + \om \|\phi\|_{L^2}^2 = c^2 \phi(0)^2=c^2.
 	$$
 	This shows that $\phi$ is a minimizer for $\min_{u\in H^s} J_\om[u]$ and so any function in the form \eqref{220} is one as well. Also,  
 	\begin{equation}
 	\label{62}
 	c^2(\om)= \f{\|(-\De)^{\f{s}{2}} \cg_s^\om\|_{L^2}^2 + \om \|\cg_s^\om\|_{L^2}^2 }{(\cg_s^\om(0))^2}=
 	\left(\int_{\rn} \f{1}{(2\pi|\xi|)^{2s}+\om} d\xi\right)^{-1}.  
 	\end{equation}
 \end{proof}
We now state a result that describes the solutions of \eqref{30}. 
\begin{lemma}
	\label{le:60} 
	The non-trivial solutions to \eqref{30}, with $\phi(0)>0$  are given by 
	\begin{equation}
	\label{250} 
	\hat{\phi}(\xi)= \left(\int_{\rn} \f{1}{(2\pi|\xi|)^{2s}+\om}  d\xi\right)^{-(1+\f{1}{2\si})} \f{1}{(2\pi|\xi|)^{2s}+\om}. 
	\end{equation}
\end{lemma}
 \begin{proof}
 	We can proceed as in the proof of Lemma \ref{le:40} to see that 
 	$$
 	\hat{\phi}(\xi)=|\phi(0)|^{2\si} \phi(0) \f{1}{(2\pi|\xi|)^{2s}+\om} . 
 	$$
 	In order to determine   $\phi(0)$, we apply the inverse Fourier transform to obtain an equation for it as follows 
 	$$
 	\phi(0)=\int_{\rn} \hat{\phi}(\xi) d\xi = |\phi(0)|^{2\si} \phi(0) \int_{\rn} \f{1}{(2\pi|\xi|)^{2s}+\om}  d\xi
 	$$
 	It follows that 
 	$$
 	|\phi(0)|^{2\si}=\left(\int_{\rn} \f{1}{(2\pi|\xi|)^{2s}+\om}  d\xi\right)^{-1},
 	$$
 	which proves the claim. 
 \end{proof}
 \begin{rmk}
 	\label{rmk:2}
 	Note that the operator $\cl_\pm$ have the form 
 	\begin{eqnarray*}
 		\cl_- &=&  (-\De)^s +\om -  |\phi(0)|^{2\sigma}\delta_0=(-\De)^s +\om -  c^2(\om) \delta_0\\
 			\cl_+ &=&  (-\De)^s +\om -  (2\si+1) c^2(\om) \delta_0. 
 	\end{eqnarray*}
 \end{rmk}
 \subsection{The spectrum of $(-\De)^s +\om -  \mu \delta_0$}
 \label{sec:3.4} 
 In this section, we develop some tools to study the bottom of the spectrum of the operators 
 $(-\De)^s +\om -  \mu \delta_0$, depending on the value of $\mu$. More specifically, we have the following result. 
 \begin{proposition}
 	\label{prop:23} 
 	Let $s>\f{n}{2}, \om>0$  and $L_\mu=(-\De)^s +\om -  \mu \delta_0$ be the self-adjoint operator introduced in Lemma \ref{le:12}. Then, 
 	\begin{itemize}
 		\item If $\mu<c^2(\om)$, the operator $L_\mu$ has one simple negative eigenvalue, 
 		$-\la_{\om, \mu}<0$, with eigenfunction $\Psi_0: \widehat{\Psi}_0(\xi)=\f{1}{(2\pi |\xi|)^{2s}+\om+\la_\om}$. For the rest of the spectrum 
 		$$
 		\si(L_\mu)\setminus \{-\la_{\om, \mu}\}\subset [\om,\infty).
 		$$
 		In particular, $L_\mu|_{\{\Psi_0\}^\perp} \geq \om$. 
 		\item If $\mu=c^2(\om)$, $L_\mu\geq 0$, $0$ is a simple eigenvalue and the rest of the spectrum, there is 
 		$
 		\si(L_\mu)\setminus \{0\}\subset [\om,\infty).
 		$
 		In particular, $L_\mu|_{\{\Psi_0\}^\perp} \geq \om$. 
 		\item If $\mu>c^2(\om)$, there is a simple eigenvalue $\la_\mu\in (0, \om)$, with eigenfunction \\ 
 		$\Psi_0: \widehat{\Psi}_0(\xi)=\f{1}{(2\pi |\xi|)^{2s}+\om-\la_\om}$ and 
 		$\si(L_\mu)\setminus \{\la_\mu\}\subset [\om,\infty)$. 
 		In particular,  $L_\mu|_{\{\Psi_0\}^\perp} \geq \la_\mu>0$. 
 	\end{itemize}
 \end{proposition}
 \begin{proof}
 	Assume first $\mu>c^2$. We would like to formally analyze the eigenvalue problem associated with the lowest eigenvalue of $L_\mu$. So, we are looking for $f\neq 0, f\in D(L_\mu)$, 
 	so that $L_\mu f=-\la f$ for some $\la>0$. This is the equation 
 	\begin{equation}
 	\label{260} 
 	((-\De)^s +\om+\la) f= \mu f(0) \de_0.
 	\end{equation}
 	Arguing as in the proof of Lemma \ref{le:40}, by taking Fourier transform etc., we find that all possible solutions are in the form
 	$$
 	\hat{f}(\xi)= \f{\mu f(0)}{(2\pi |\xi|)^{2s}+\om+\la}. 
 	$$
 	Clearly, $f\in D(L_\mu)$ and we need to see that there exists $\la>0$, so that it solves \eqref{260}. To this end, we have 
 	$$
 	f(0)=\int_{\rn} \hat{f}(\xi) d\xi =\mu f(0) \int_{\rn} \f{1}{(2\pi |\xi|)^{2s}+\om+\la}d\xi. 
 	$$
 	As we seek non-trivial solutions $f$ (and hence $f(0)\neq 0$), this amounts to finding $\la$,  so that for the given $\om$, we have 
 	\begin{equation}
 	\label{270} 
 	\mu \int_{\rn} \f{1}{(2\pi |\xi|)^{2s}+\om+\la}d\xi=1.
 	\end{equation}
 	We claim that under the condition $\mu>c^2$, there is exactly one solution $\la=\la_{\om, \mu}\in (0, \infty)$. Indeed,  consider the continuous and decreasing function 
 	$$
 	h(\la):=\mu \int_{\rn} \f{1}{(2\pi |\xi|)^{2s}+\om+\la}d\xi-1.
 	$$
 	Computing its limits at the ends of the interval
 	$$
 	\lim_{\la\to 0+} h(\la) = \mu \int_{\rn} \f{1}{(2\pi |\xi|)^{2s}+\om}d\xi-1=\f{\mu}{c^2} -1>0, 
 	\lim_{\la\to +\infty} h(\la)=-1, 
 	$$
 	implies that there is an unique eigenvalue $\la_{\om,\mu}>0$. Moreover, the corresponding eigenfunction is, up to a multiplicative constant 
 $$
 \widehat{\Psi}_0(\xi)=\f{1}{(2\pi |\xi|)^{2s}+\om+\la_{\om,\mu}}
 $$
 We now prove the statement about  the rest of the spectrum. Consider the spectral decomposition of the self-adjoint operator $L_\mu$. Assume for a contradiction that for any $\de>0$, we have that  $\si(L_\mu)\cap (-\la_{\om,\mu}+\de, \om-\de)\neq \emptyset$. Let 
 $\Psi\in Image ({\mathbb P}_{(-\la_{\om,\mu}+\de, \om-\de)})$ (i.e. $\Psi={\mathbb P}_{(-\la_{\om,\mu}+\de, \om-\de)}\Psi$) and then normalize it, that is  $\|\Psi\|_{L^2}=1$. As 
 $\Psi_0(0)=\int_{\rn} \f{1}{(2\pi |\xi|)^{2s}+\om+\la_{\om,\mu}} d\xi>0$, consider the well-defined element of $D(L_\mu)$, 
 $$
 \tilde{\Psi}:=\Psi-\f{\Psi(0)}{\Psi_0(0)} \Psi_0.
 $$
 Note that $\tilde{\Psi}(0)=0$, so according to \eqref{112}, we have,
 $$
 \dpr{L_\mu \tilde{\Psi}}{\tilde{\Psi}}=\|(-\De)^{\f{s}{2}} \tilde{\Psi}\|_{L^2}^2+ 
 \om \| \tilde{\Psi}\|_{L^2}^2\geq \om \| \tilde{\Psi}\|_{L^2}^2\geq \om \|\Psi\|_{L^2}^2 = \om.
 $$
 where  we have used  that $\Psi\perp \Psi_0$, and hence $\|\tilde{\Psi}\|_{L^2}^2=\|\Psi\|_{L^2}^2+ \f{\Psi^2(0)}{\Psi_0^2(0)} \|\Psi_0\|_{L^2}^2\geq \|\Psi\|_{L^2}^2=1$. 
 
 On the other hand, again by $\Psi\perp \Psi_0, L_\mu \Psi\perp \Psi_0$, and the properties of the spectral projections, 
 \begin{eqnarray*}
 \dpr{L_\mu \tilde{\Psi}}{\tilde{\Psi}}=\dpr{L_\mu \Psi}{\Psi}+\f{\Psi^2(0)}{\Psi_0^2(0)}\dpr{L_\mu \Psi_0}{\Psi_0}\leq (\om-\de) -\la_{\om,\mu} \f{\Psi^2(0)}{\Psi_0^2(0)}\leq \om-\de.
 \end{eqnarray*} 
 Clearly, the two estimates that we have obtained for $\dpr{L_\mu \tilde{\Psi}}{\tilde{\Psi}}$ are contradictory, which is due to the assumption $\si(L_\mu)\cap (-\la_{\om,\mu}, \om-\de)\neq \emptyset$. Thus, $\si(L_\mu)\cap (-\la_{\om,\mu}, \om)= \emptyset$ or $\si(L_\mu)\setminus \{-\la_{\om,\mu}\}\subset [\om,\infty)$, which was the claim. 
 
 The proof for $\mu=c^2$ is along similar lines. Indeed, for any test function $\Psi\in H^s$, we have 
 $$
 \dpr{L_\mu \Psi}{\Psi}=\|(-\De)^{\f{s}{2}} \Psi\|_{L^2}^2+ 
 \om \| \Psi\|_{L^2}^2-c_s^2 |\Psi(0)|^2\geq 0,
 $$
 by the definition of $c^2=\inf J_\om[\Psi]$. Hence, $L_\mu\geq 0$. 
 Furthermore, by direct inspection $L_\mu[\cg_\om^s]=0$, whence $0$ is an eigenvalue (and it would have to be at the bottom of the spectrum). Finally, $\si(L_\mu)\setminus \{0\} \subset [\om,\infty)$ is shown in the exact same way as in the case $\mu>c^2$. 
 
 For the case $\mu<c^2$, we can similarly identify an unique $\la_{\om,\mu}\in (0, \om)$, so that 
 $$
 \mu \int_{\rn} \f{1}{(2\pi |\xi|)^{2s}+\om-\la}d\xi=1.
 $$
 This $\la_{\om,\mu}>0$ is an eigenvalue for $L_\mu$, with eigenfunction, 
 $\Psi_0: \widehat{\Psi}_0(\xi)=\f{1}{(2\pi |\xi|)^{2s}+\om-\la}$. Moreover, $\si(L_\mu)\setminus\{\la_{\om,\mu}\} \subset [\om,\infty)$ is proved in the same fashion as above. 
 \end{proof}
 As a direct consequence of the results of Proposition \ref{prop:23} and Remark \ref{rmk:2}, we have the following corollary. 
 \begin{corollary}
 	\label{cor:2}
 	Let $s>\f{n}{2}$, $\om>0$, $\si>0$. Then, 
 	\begin{itemize}
 		\item $\cl_-\geq 0$, $0$ is a simple eigenvalue, with eigenfunction $\cg_s^\om$ and 
 		$$
 		\si(\cl_-)\setminus \{0\}\subset [\om,\infty)
 		$$
 		Also, $\cl_-|_{\{\cg_s^\om\}^\perp}\geq \om$. 
 		\item $\cl_+$ has a simple negative eigenvalue, with an eigenfunction $\Psi_0$. 
 		Also, 
 		$$
 		\cl_+|_{\{\Psi_0\}^\perp}\geq \om>0.
 		$$
 		
 	\end{itemize}
 	
 \end{corollary}

\section{Stability of the waves} 
\label{sec:4} 
 In this section, we identify the regions of stability for the waves. We start with a short introduction in the theory of the Hamiltonian instability index, as developed in \cite{KKS, KKS2, KP}.  
 \subsection{The Hamiltonian instability index theory} 
 We are concerned with a Hamiltonian eigenvalue problem of the form 
 \begin{equation}
 \label{s:12} 
  \ci \ck f=\la f, 
 \end{equation}
where $\ci^*=-\ci, \ck^*=\ck$, $\ci$ is bounded and invertible, so that $\ci^{-1}: Ker(\ck)\to Ker(\ck)^\perp$. 

We would analyze   the number of unstable eigenvalues of the eigenvalue  problem \eqref{s:12}. To this end, we assume that the Morse index of $\ck$ is finite, that is $n(\ck)=\# \{\mu\in \si_{p.p.}(\ck), \mu<0\}<\infty$ and $dim(Ker(\ck))<\infty$, say $Ker(\ck)=span\{\psi_j, j=1, \ldots, N\}$. Introduce a scalar matrix $\cd$, with entries\footnote{Note that since  $\ci^{-1}: Ker(\ck)\to Ker(\ck)^\perp$, the operator $\ck^{-1}$ is well-defined on $\ci^{-1} \psi$}
$$
\cd_{i j}=\dpr{\ck^{-1} \ci^{-1} \psi_i}{\ci^{-1} \psi_j}
$$
 Then, according to \cite{KKS, KKS2, KP}, we have the following formula 
 \begin{equation}
 \label{se:10} 
 k_r+ k_c+ k_0^{\leq 0}= n(\cl)-n(\cd), 
 \end{equation}
 where $k_r$ is the number of real and positive solutions $\la$ in \eqref{s:12},   accounting for the real unstable modes, $k_c$ is the number of solutions $\la$ in \eqref{s:12} with positive real part, while  finally $k_0^{\leq 0}$ is the number of the dimension of the marginally  stable directions, corresponding to purely imaginary eigenvalue with negative Krein index. Note that by Hamiltonian symmetry considerations, both $k_c,   k_0^{\leq 0}$ are even non-negative integers. 
 
 A very immediate corollary of the considerations above is the following statement, which is often referred to as the Vakhitov-Kolokolov stability condition. 
 \begin{corollary}
 	\label{s:10} Let  $\ck$ be self-adjoint, with $n(\ck)=1, dim(Ker(\ck))=1$, say  $Ker(\ck)=span\{\Psi\}$. Assume that $\ci$ also satisfy the assumptions listed above. Then, the Hamiltonian eigenvalue problem \eqref{s:12} is stable if and only if 
 	\begin{equation}
 	\label{s:22} 
 	\dpr{\ck^{-1} \ci^{-1} \Psi}{\ci^{-1} \Psi}<0.
 	\end{equation}
 \end{corollary}
 Indeed, in such a setup, the matrix $\cd$ is one dimensional matrix. Also,  the right-hand side of \eqref{se:10} is either $0$ or $1$, whence $k_r=n(\cl)-n(\cd)=1-n(\cd)$ and stability is equivalent to $n(\cd)=1$, which is exactly the condition \eqref{s:22}. 
 
 \subsection{Instability index count for \eqref{412}}
In our specific case, we need to apply the instability index counting theory to the eigenvalue problem \eqref{412}. Recall that  $\cj^*=-\cj=\cj^{-1}$, while $\cl=\left(\begin{array}{cc}
	\cl_- & 0 \\ 0 & \cl_+
\end{array}\right)$, whence 
$$
n(\cl)=n(\cl_+)+n(\cl_-)=1+0=1,
$$
due to the results of Corollary \ref{cor:2}. Also, again by the description in Corollary \ref{cor:2}, 
$$
Ker(\cl)=\left(\begin{array}{c}
Ker(\cl_-) \\ 0
\end{array}\right)+\left(\begin{array}{c}
0 \\ Ker(\cl_+)
\end{array}\right)=span\left(\begin{array}{c}
\phi_\om \\ 0
\end{array}\right).
$$
 It follows that Corollary \ref{s:10} is applicable to the eigenvalue problem \eqref{412}, and in fact the spectral stability of it is equivalent to  the condition 
 \begin{equation}
 \label{s:50} 
 \dpr{\cl_+^{-1} \phi_\om}{\phi_\om}<0.
 \end{equation}
 Since, $\phi_\om=c \cg_s^\om$, it suffice to compute 
 $\dpr{\cl_+^{-1} \cg^s_\om}{\cg^s_\om}$. We accomplish this in the following proposition. 
 \begin{proposition}
 	\label{prop:s:26} 
 	Let $n\geq 1$, $\om>0$, $\si>0$ and $s>\f{n}{2}$. Then, 
 	$$
 	sgn\dpr{\cl_+^{-1} \phi_\om}{\phi_\om}=sgn \dpr{\cl_+^{-1} \cg^s_\om}{\cg^s_\om}=sgn \left(\si-\f{2s-n}{n} \right). 
 	$$
 	In particular, the waves $\phi_\om$ are spectrally stable if and only if 
 	$$
 0<	\si<\f{2 s}{n} - 1. 
 	$$
 \end{proposition}
 \begin{proof}
 	We first need to find $\cl_+^{-1} \cg^s_\om$. That is, we need to solve $\cl_+\psi=\cg^s_\om$. Based on the formula \eqref{70} however, we need to solve  
 	$$
 	\cg^s_\om=\cl_+\psi=((-\De)^s+\om) g
 	$$
 	whence,  we can actually find $g$ pretty easily by taking Fourier transform. Namely, 
 	$$
 	((2\pi|\xi|)^{2s}+\om) \hat{g}(\xi)= \widehat{\cg^s_\om}(\xi)=\f{1}{(2\pi|\xi|)^{2s}+\om}. 
 	$$
 It follows that 
 $$
 \hat{g}(\xi)= \f{1}{((2\pi|\xi|)^{2s}+\om)^2}, 
 $$	
 or equivalently $g=\cg^s_\om*\cg^s_\om$. We can now proceed to find $\psi$ from \eqref{80}. Namely, taking into account that $\cl_+=(-\De)^s+\om-(2\si+1)c^2$, we compute 
 $$
 \psi=g+ (2\si+1) c^2 \f{g(0)}{1-(2\si+1) c^2 \cg^s_\om(0)} \cg^s_\om
 $$
 	Note however that $g(0)=\cg^s_\om*\cg^s_\om(0)=\|\cg^s_\om\|_{L^2}^2$. Also, according to \eqref{62}, $c_s^2 \cg^s_\om(0)=1$, so 
 	$$
 	\psi=\cg^s_\om*\cg^s_\om -\f{2\si+1}{2\si} 
 	\f{\int_{\rn} \f{1}{((2\pi|\xi|)^{2s}+\om)^2} d\xi}{\int_{\rn} \f{1}{(2\pi|\xi|)^{2s}+\om} d\xi} \cg^s_\om.
 	$$
 	So 
 	\begin{eqnarray*}
 	\dpr{\cl_+^{-1} \cg^s_\om}{\cg^s_\om} &=& \dpr{\psi}{\cg^s_\om}=\dpr{\cg^s_\om*\cg^s_\om}{\cg^s_\om}-\f{2\si+1}{2\si} 
 	\f{\int_{\rn} \f{1}{((2\pi|\xi|)^{2s}+\om)^2} d\xi}{\int_{\rn} \f{1}{(2\pi|\xi|)^{2s}+\om} d\xi} \dpr{\cg^s_\om}{\cg^s_\om}=\\
 	&=& \int_{\rn} \f{1}{((2\pi|\xi|)^{2s}+\om)^3} d\xi-\f{2\si+1}{2\si} 
 	\f{\left(\int_{\rn} \f{1}{((2\pi|\xi|)^{2s}+\om)^2} d\xi\right)^2}{\int_{\rn} \f{1}{(2\pi|\xi|)^{2s}+\om} d\xi}. 
 	\end{eqnarray*}
 	So, it remains to compute 
 	$$
 	\int_{\rn} \f{1}{((2\pi|\xi|)^{2s}+\om)^j }d\xi, j=1,2,3.
 	$$
 	which we have done in the Appendix, see Proposition \ref{int}. More specifically, substituting the formulas \eqref{b:12}, \eqref{b:14}, \eqref{b:16} in the expression for 
 	$\dpr{\cl_+^{-1} \cg^s_\om}{\cg^s_\om}$, we obtain 
 	\begin{eqnarray*}
 	\dpr{\cl_+^{-1} \cg^s_\om}{\cg^s_\om} &=& \f{\pi |\sn| \om^{\f{n}{2s}-3} }{ 2 (2\pi)^n \sin(\f{n\pi}{2s})} \left(\left(1-\f{n}{2s}\right) \left(2-\f{n}{2s}\right) -  \f{2\si+1}{\si} \left(1-\f{n}{2s}\right)^2   \right)=\\
 	&=& \f{n \pi |\sn| \om^{\f{n}{2s}-3} }{ 4 s \si (2\pi)^n  \sin(\f{n\pi}{2s})} \left(1-\f{n}{2s}\right) \left(\si-\f{2s-n}{n} \right). 
 	\end{eqnarray*}
 	Note that, as $s>\f{n}{2}$, only the last term in the expression changes sign over the parameter space. We have this established  Proposition \ref{prop:s:26} in full. 
 \end{proof}
 The above spectral properties of the operator $\cl_{\pm}$ we have one last stop before arriving at the orbital stability of the wave, that is we need to argue the coerciveness of $\cl_{\pm}$ on the space $H^s(\rn)$.  To that end we have the following proposition
 \begin{proposition}
    \label{vk}
  		Let $s>\frac n2,\om>0,$ $\dpr{\cl_+^{-1} \phi_\om }{\phi_\om}<0$. Then, the operator $\cl_+$ is coercive on $\{\phi_\om\}^\perp$. That is, there exists $\de>0$, so that for all 
  		\begin{equation}
  		\label{1072}
  		\dpr{\cl_+ \Psi}{\Psi} \geq \de \|\Psi\|_{H^{s}}^2,\ \  \forall \Psi\perp \phi_\om.
  		\end{equation}
\end{proposition}
 \begin{proof}
 This is a version of a well-known lemma in the theory, see for example Lemma 6.7 and Lemma 6.9 in \cite{Pava}.  Recall that we have already showed $Ker[\cl_+]=\{0\}$ and $n(\cl_+)=1$. According\footnote{and this is already explicit in a much  earlier work by Weinstein} to  Lemma 6.4, \cite{Pava} under these  conditions for $\cl_+$ 
 we have that for any $g\perp \phi_\om$, 
 \begin{equation}
 \label{1108}
 \dpr{\cl_+ g}{g}\geq 0.
 \end{equation}
  Consider the associated constrained minimization problem 
  \begin{equation}
  \label{1105} 
  \inf\limits_{\|f\|=1, f\perp \phi_\om} \dpr{\cl_+ f}{f} 
  \end{equation} 
  and set 
 $$
 \al:=\inf\{\dpr{\cl_+ f}{f}: f\perp \phi_\om, \|f\|_{L^2}=1\}\geq 0.
 $$
We will show that $\al>0$. Assume for a contradiction that $\al=0$. 

  Take a minimizing sequence $f_k: \|f_k\|=1, f_k\perp \phi_\om$, 
 so that 
 $$
 \al=\lim_k \dpr{\cl_+ f_k}{f_k}=\lim_k [\|(-\De)^{\f{s}{2}} f_k\|^2 +\om- (2\sigma+1)c^2|f_k(0)|^2].
 $$
 However, by Sobolev embedding and the Gagliardo-Nirenberg's inequalities  (recall $\|f_k\|_{L^2}=1$), we have that for all $\be: \f n2<\beta<s$ and for all $\eps>0$, 
 $$
 |f(0)|\leq \|f\|_{L^\infty} \leq C_\be (\|f\|_{\dot{H}^\be}+C \|f\|_{L^2} \leq C_\be \|f\|_{\dot{H}^s}^{\f{\be}{s}} \|f\|_{L^2}^{1-\f \be s}+C \|f\|_{L^2}\leq \eps \|f\|_{\dot{H}^s}  + C_\eps  \|f\|_{L^2}. 
 $$
 Applying this estimate, we obtain a lower bound for $\dpr{\cl_+ f_k}{f_k}$ (recall $\|f_k\|_{L^2}=1$), as follows 
 $$
 \dpr{\cl_+ f_k}{f_k}\geq \f{1}{2} \|(-\De)^{\f{s}{2}} f_k\|^2 - C.
 $$
 Since, $\al=\lim_k \dpr{\cl_+ f_k}{f_k}$, this implies that $\sup_k \|(-\De)^{\f{s}{2}} f_k\|^2<\infty$. This means that we can select a subsequence of $\{f_k\}$ (denoted by the same), so that $f_k$ converges weakly to $f\in H^s(\rn)$. In addition, by the Sobolev embedding $H^s(\rn)\hookrightarrow C^\ga(\rn), \ga<s-\f{n}{2}$, we can, as we have done previously, without loss of generality assume that $f_n\rightrightarrows f$ on the compact subsets of $\rn$. In particular, $\lim_k f_k(0)=f(0)$. Note that by the weak convergence, $\dpr{f}{\phi_\om}=\lim_k \dpr{f_k}{\phi_\om}=0$, so $f\perp \phi_\om$ and 
  \begin{equation}
  \label{1118}
 \liminf_k \|(-\De)^{\f{s}{2}} f_k\|^2\geq \|(-\De)^{\f{s}{2}} f\|^2, \ \  \|f\|_{L^2}\leq \liminf \|f_k\|_{L^2}=1.
 \end{equation}
 It follows that 
 \begin{equation}
 \label{1112}
  \dpr{\cl_+ f}{f}\leq \liminf_k \dpr{\cl_+ f_k}{f_k} =0.
 \end{equation}
 But by \eqref{1108}, and since $f\perp \phi_\om$, we have that $\dpr{\cl_+ f}{f}\geq 0$. It follows that $0=\dpr{\cl_+ f}{f}=\lim_k \dpr{\cl_+ f_k}{f_k}$. But this means that all inequalities in \eqref{1118} and \eqref{1112} are equalities and in particular
 $$
 \lim_k \|(-\De)^{\f{s}{2}} f_k\|_{L^2}=\|(-\De)^{\f{s}{2}}  f\|_{L^2}, \lim_k \|f_k\|_{L^2}=\|f\|_{L^2}.
 $$
This last identities, in addition to the $H^s$ weak convergence $f_k$ to $f$,  implies strong convergence, that is  $\lim_k \|f_k-f\|_{H^s}=0$. In particular, $\|f\|_{L^2}=\lim_k \|f_k\|_{L^2}=1$. In other words, $f$ is a minimizer for the constrained minimization problem \eqref{1105}. 
   Write the Euler-Lagrange equation for $f$ 
 \begin{eqnarray}
 \label{1110} 
 \cl_+ f = d  f + c \phi_\om
   .\end{eqnarray}
 Taking dot product with $f$ and taking into account $\dpr{\cl_+ f}{f}=0$, $f\neq 0$ and $f\perp \phi_\om$ implies that $d=0$. This means that $f=c \cl_+^{-1} \phi_\om$. But then, 
 $$
 0=\dpr{\cl_+ f}{f}=c^2 \dpr{\cl_+^{-1} \phi_\om}{\phi_\om}.
 $$
 Since $\dpr{\cl_+^{-1} \phi_\om}{\phi_\om}\neq 0$, it follows $c=0$. But then, since $Ker[\cl_+]=\{0\}$, \eqref{1110} implies that  $f=0$, which is a contradiction. 
 Thus, we have shown that $\al>0$. As a consequence, 
 \begin{eqnarray}
 \label{1120} 
 \dpr{\cl_+ \Psi}{\Psi}\geq \al \|\Psi\|_{L^2}^2, \ \ \forall \Psi\perp \phi_\om.
 \end{eqnarray}
Note that \eqref{1072} is however stronger than \eqref{1120}, as it involves 
$\|\cdot\|_{H^s}$ norms on the right-hand side. Nevertheless, we show that it is relatively straightforward to deduce it from \eqref{1120}. 
Indeed, assume for a contradiction in \eqref{1072}, that $g_k: \|g_k\|_{H^s}=1, g_k\perp \phi_\om$, so that  $\lim_k \dpr{\cl_+ g_k}{g_k}=0$. 

Taking into account \eqref{1120}, 
this is only possible if $\lim_k \|g_k\|_{L^2}=0$. So, 
$$1
=\lim_k[\|(-\De)^{\f{s}{2}} g_k\|_{L^2}^2+ \|g_k\|_{L^2}^2]=\lim_k \|(-\De)^{\f{s}{2}} g_k\|_{L^2}^2.
$$
 Note that by \eqref{017}, we have that for all $0<\de<s-\f n2$, we have that 
 $$
 |g_k(0)|\leq \|g_k\|_{L^\infty} \leq C(\|g_k\|_{\dot{H}^{\f n2 + \de}}+ \|g_k\|_{\dot{H}^{\f n2 - \de}})\leq 
 C(\|g_k\|_{\dot{H}^s}^{\f{\f{n}{2}+\de}{s}} \|g_k\|_{L^2}^{1-\f{\f{n}{2}+\de}{s}}+ \|g_k\|_{\dot{H}^s}^{\f{\f{n}{2}-\de}{s}} \|g_k\|_{L^2}^{1-\f{\f{n}{2}-\de}{s}},
 $$
 whence $\lim_k \|g_k(0)|=0$. 
 But then, we achieve a contradiction
 $$
 0=\lim_k \dpr{\cl_+ g_k}{g_k} = \lim_k [\|(-\De)^{\f{s}{2}} g_k\|_{L^2}^2+\om \|g_k\|^2 - (2\sigma+1)c_s^2|g_k(0)|^2]=1,
 $$
 \end{proof}

 \subsection{Orbital stability }
  In this section, we prove that the spectrally stable solutions are in fact orbitally stable. That is, we consider the case $0<\si<\f{2s}{n}-1$. 
  \begin{proposition}
 	\label{prop:11} 
      Let $\om>0$, $n\geq 1, s>\f n2$, $0<\si<\f{2s}{n}-1$ and the key assumptions (1), (2) are satisfied. 
 Then $e^{ i \om t} \phi_\om$ is orbitally stable solution of \eqref{20}. 
 \end{proposition}
 \begin{proof}
 	Let us outline first what  the consequences of our assumptions are. By Proposition \ref{prop:s:26}, we have that $\dpr{\cl_+^{-1} \phi_\om}{\phi_\om}<0$, which by Proposition \ref{vk} means that the coercivity estimate \eqref{1072} holds. By Corollary \ref{cor:2}, $Ker(\cl_+)=\{0\}$, that is the wave $\phi_\om$ is non-degenerate. 
 	
 	We now concentrate on the orbital sdtability. Our proof is  by a contradiction argument. That is, there  is  $\epsilon_0>0$ and a sequence of initial data $u_k:\lim_k \|u_k-\phi\|_{H^s(\mathbb{R}^n)}=0$, 
  so that 
  \begin{equation}
  \label{contradict}
      \sup_{0\leq t<\infty}\inf_{\theta\in \rone }\|u_k(t,\cdot)-e^{-i\theta}\phi\|_{H^s}\geq \epsilon_0. 
  \end{equation}
 
Using the conserved quantities \eqref{conserve} we define new conserved quantity $$
\mathcal{E}[u]:=E[u]+\f{\om}{2}M[u],
$$ $$\eps_k:=|\mathcal{E}[u_k(t)]-\mathcal{E}[\phi_\om]]|+|M[u_k(t)]-M[\phi_\om]]|,$$ and for all  $\eps>0$,  $$
 t_k:=\sup\{\tau:\sup_{0<t<\tau} \|u_k(t)-\phi\|_{H^s(\mathbb{R}^n)}<\eps\}
. $$ 
Note that $\eps_k$ is conserved and  $\lim_k\eps_k=0$ and by the assumption that we have local \\ well-posedness  $t_k>0$. 

Consider  $t\in(0,t_k)$ and let $u_k=v_k+iw_k$ and $\|w_k(t)\|_{H^s(\mathbb{R}^n)}\leq2\|u_k-\phi\|_{H^s(\mathbb{R}^n)}<\eps.$ This leads to the definition of the modulation parameter $\theta_k(t)$ such that $w_k+\sin{\theta_k(t)}\phi\perp \phi$ that is 

\begin{equation}
   \label{1640}
     -\sin(\theta_k(t))\|\phi\|=\langle w_k(t),\phi\rangle.
\end{equation}
By Cauchy-Schwartz we have $|\langle w_k(t),\phi\rangle|\leq\eps\|\phi\|_{L^2}$ and this means there is an  unique small solution $\theta_k(t)$ of \ref{1640}, with $|\theta_k(t)|\leq \eps.$ Also  $$
  \|u_k(t,\cdot)-e^{-i\theta_k(t)}\phi\|_{H^s}\leq 
  \|u_k(t, \cdot)- \phi\|_{H^s}+ |e^{-i \theta_k(t)}-1|\|\phi\|_{H^s}\leq  C(\|\phi\|_{H^s})\eps,
  $$ 
Now define 
$$
  T_k:=\sup\{\tau:\sup_{0<t<\tau} \|u_k(t, \cdot)-
  	e^{-i\theta_k(t)}\varphi(\cdot)\|_{H^s(\mathbb{R}^n)}<2C\eps\}
.  $$

Clearly $0<t_k<T_k.$ From this we see that to get contradiction of \eqref{contradict} it is enough to show that for all $\eps>0$ and large $k,T_k=\infty.$ To that end let $t\in(0,T_k)$ write $$
\psi_k=u_k-e^{-i\theta_k(t)}\phi=v_k+iw_k-e^{-i\theta_k(t)}\phi
$$ and decompose into real and imaginary part of $\psi_k$ and projecting on  $\left(\begin{array}{cc}
  \phi  \\ 0 
  \end{array}\right)$ yield 
  
   \begin{equation}
     \label{1021}
    \left(\begin{array}{cc}
v_k(t,\cdot)-\cos(\theta_k(t))\phi  \\ w_k(t,\cdot)+\sin(\theta_k(t))\phi 
\end{array}\right)=\mu_k(t)\left(\begin{array}{cc}
\phi  \\ 0 
\end{array}\right)+\left(\begin{array}{cc}
\eta_k(t,\cdot)  \\ \zeta_k(t,\cdot)
\end{array}\right),\ \ \left(\begin{array}{cc}
\eta_k(t,\cdot)  \\ \zeta_k(t,\cdot)
\end{array}\right)\perp \left(\begin{array}{cc}
\phi  \\ 0 
\end{array}\right)
. \end{equation}
  By the choice of $\theta_k$ we have $\zeta_k\perp\phi,$ and from the above decomposition we also have $\eta_k\perp\phi.$ So taking the $L^2$ norm of \eqref{1021} we have 
  
  \begin{equation}
\label{1028} 
|\mu_k(t)|^2 \|\phi\|^2_{L^2} + \|\eta_k(t)\|^2_{L^2}+ \|\zeta_k(t)\|^2_{L^2}=\|\psi_k(t)\|^2_{L^2}\leq 4C^2 \eps^2.
\end{equation}
  Next we take advantage of the two conserved quantities, to that end we consider the mass 
  
  \begin{eqnarray*}
	M[u_k(t)] &=& \int_{\rn} |e^{-i \theta_k(t)} \phi+\psi_k(t)|^2 dx=M[\phi]+\|\psi_k(t, \cdot)\|_{L^2}^2+ 2\int_{\rn} \phi(x) \Re  [e^{-i \theta_k(t)} \psi_k(t, x)] dx.\\
	&=&M[\phi]+\|\psi_k(t, \cdot)\|_{L^2}^2+ 2  \mu_k(t) \cos(\theta_k(t)) \|\phi\|^2
\end{eqnarray*}
  
  Here used the fact that $w_k+\sin{\theta_k(t)}\phi\perp \phi$ and $\eta_k\perp\phi.$ Solving for $\mu_k(t)$ and since $|\theta_k|$ is very small and $\|\psi_k(t, \cdot)\|_{L^2}\leq 2C \eps$, in $t: 0<t<T_k$  we have 
  
  \begin{equation}
\label{1048} 
|\mu_k(t)|\leq \f{|M[u_k(t)]-M[\phi]|+ \|\psi_k(t, \cdot)\|_{L^2}^2}{2\cos(\theta_k(t) \|\phi\|^2}
\leq C(\eps_k+\|\psi_k(t, \cdot)\|_{L^2}^2)\leq C(\eps_k+\eps^2). 
\end{equation}
  
  Now we will expand $\mathcal{E}[u_k(t)]-\mathcal{E}[\phi]$ but first for any small perturbations of the wave $\alpha_1+i \alpha_2\in H^s(\rn)$ and using \eqref{30} we have 
  
  \begin{equation}
\label{1104}
E[\phi+(\alpha_1+i \alpha_2)]-E[\phi] = \f{1}{2}[\dpr{\cl_+ \alpha_1}{\alpha_1}+ 
\dpr{\cl_- \alpha_2}{\alpha_2}]+Err[\alpha_1,\alpha_2], 
\end{equation}
where 
\begin{align*}
    &|Err[\alpha_1,\alpha_2]|\leq C|((\phi(0)+\alpha_1(0))^2+ \alpha_2^2(0))^{\sigma+1}-\phi(0)^{2\sigma+2}\\
    &-({2\sigma+2})\phi(0)^{2\sigma+1}\alpha_1(0)-\f{(2\sigma+2)(2\sigma+1)}{2}\phi^{2\sigma}(0)\alpha_1^2(0)-(2\sigma+2) \phi^{2\sigma}(0)\alpha_2^2(0)  |\\
    &\leq C(\|\phi\|_{L^\infty})(|\alpha_1(0)|+|\alpha_2(0)|)^{\min(2\sigma+2,3)}.
\end{align*}
Apply this expansion \eqref{1104} to 
$$
\alpha_1+i \alpha_2=e^{i \theta_k(t)} \psi_k=\left[\cos(\theta_k)(\mu_k\phi+\eta_k)-\sin(\theta_k) \zeta_k\right]+
i \left[\cos(\theta_k)\zeta_k+\sin(\theta_k)(\mu_k\phi+\eta_k)\right].
$$
From \eqref{1028}, we see that $\|\alpha_1\|_{H^s}+\|\alpha_2\|_{H^s}\leq C\eps$, so we can bound the contribution of $|Err[\alpha_1, \alpha_2]|$ as follows 
\begin{equation}
\label{1205} 
|Err[\alpha_1, \alpha_2]|\leq C\eps^{\min (2\sigma,1)}  (\|\alpha_1\|_{H^s}^2+ \|\alpha_2\|_{H^s}^2).
\end{equation}

By the Sobolev embeddings, $\cl_-\phi=0$ and $\cl_+=\cl_--2\sigma|\phi(0)|^{2\sigma}\delta$ together with \eqref{1028} and \eqref{1048} we have

\begin{eqnarray*}
& & \dpr{\cl_+ \alpha_1}{\alpha_1}\geq\dpr{\cl_+ \eta_k}{\eta_k}-C(\eps^3+\eps_k+\eps^2(\|\eta_k\|_{H^s}+\|\zeta_k\|_{H^s})+
\eps(\|\eta_k\|_{H^s}+\|\zeta_k\|_{H^s})^2)
\\
& & \dpr{\cl_- \alpha_2}{\alpha_2}\geq \dpr{\cl_- \zeta_k}{\zeta_k} -C(\eps^3+\eps_k+\eps^2(\|\eta_k\|_{H^s}+\|\zeta_k\|_{H^s})+
\eps (\|\eta_k\|_{H^s}+\|\zeta_k\|_{H^s})^2) 
.\end{eqnarray*}

Taking advantage of  the coercivity of $\cl_-$  and $\cl_+$, which was established in Proposition \ref{prop:23}, we have that for some $\ka>0$ and since $\eta_k, \zeta_k \perp \phi$ together with some algebraic manipulations yield 

\begin{equation}
 \label{1544} 
 \|\eta_k(t)\|_{H^s}^2+\|\zeta_k(t)\|_{H^s}^2\leq C (\eps^3+\eps_k),
 \end{equation}
Here $C$ is independent of $\eps$ and $k$. This implies that  $T_k^*=\infty$, since if we assume that  $T_k^*<\infty$, then 
 \begin{equation}
 \label{eps}
 2 C_0\eps=\limsup_{t\to T_k^*-} \|\psi_k(t)\|_{H^s}\leq C(|\mu_k(t)|+  \|\eta_k(t)\|_{H^s}+\|\zeta_k(t)\|_{H^s})\leq C(\eps^{\f{3}{2}} +\sqrt{\eps_k}).
 \end{equation}
 which is a contradiction, if $\eps$ is so that $C_0\eps>C \eps^{\f{3}{2}}$ and then $k$ is so large, and hence $\eps_k$ is so small,  that  $C_0\eps>C\sqrt{\eps_k}$, which certainly  contradicts \eqref{eps}.  Hence the wave is orbitally stable. 

  \end{proof}

 \appendix
 
 \section{The integrals $	\int_{\rn} \f{1}{((2\pi|\xi|)^{2s}+\om)^j }d\xi$}
 Herein, we compute the integrals that arise in the calculation of the Vakhitov-Kolokolov index in Proposition \ref{prop:s:26}. 
 \begin{proposition}
 	\label{int} 
 	For $\om>0$, we have 
 	\begin{eqnarray}
 	\label{b:12} 
 	\int_{\rn} \f{1}{(2\pi|\xi|)^{2s}+\om}d\xi &=& \f{\pi |\sn|}{(2\pi)^n}  \f{\om^{\f{n}{2s}-1}}{\sin(\f{ n\pi}{2s})} \\
 	\label{b:14}
 	 \int_{\rn} \f{1}{((2\pi|\xi|)^{2s}+\om)^2}d\xi &=& \f{\pi |\sn|}{(2\pi)^n} \left(1-\f{n}{2s}\right)  \f{\om^{\f{n}{2s}-2}}{\sin(\f{ n\pi}{2s})} \\
 	 \label{b:16}
 	 \int_{\rn} \f{1}{((2\pi|\xi|)^{2s}+\om)^3}d\xi &=& \f{\pi |\sn| }{2 (2\pi)^n} \left(1-\f{n}{2s}\right) \left(2-\f{n}{2s}\right) \f{\om^{\f{n}{2s}-3}}{\sin(\f{ n\pi}{2s})} 
 	\end{eqnarray}
 \end{proposition}
 \begin{proof}
 	We easily pass to integrals in the radial variable as follows 
 	\begin{eqnarray*}
 		\int_{\rn} \f{1}{((2\pi|\xi|)^{2s}+\om)^j }d\xi &=& |\sn| \int_0^\infty  \f{\rho^{n-1} }{((2\pi\rho)^{2s}+\om)^j }d\rho =|\sn| \f{\om^{\f{n}{2s}-j}}{(2\pi)^n} \int_0^\infty \f{\rho^{\f{n}{2s}-1}}{(\rho+1)^j}d\rho=\\
 		&=& |\sn| \f{\om^{\f{n}{2s}-j}}{(2\pi)^n} \int_{-\infty}^\infty \f{e^{t \f{n}{2s}}}{(e^t+1)^j} dt
 	\end{eqnarray*}
 So, with $a:=\f{n}{2s}\in (0,1)$, matters are clearly reduced to computing the integrals 
 $$
 \int_{-\infty}^\infty \f{e^{t a}}{(e^t+1)^j} dt, 
 $$
 	for $a\in (0,1)$, $j=1,2,3$. 
 	In order to compute this integral, we use the residue theorem formula 
 	$$
 	\int_{\ga_R} \f{e^{a z}}{(e^z+1)^j} dz=2\pi i Res\left(\f{e^{a z}}{(e^z+1)^j}, \pi i \right).
 	$$
 	where $R>>1$, and $\ga_R=\ga^1_R\cup \ga^2_R\cup \ga^3_R\cup \ga^4_R$, and the curves $\ga_r^m, m=1,2,3,4$ are given, together with their orientation as follows,  
 	\begin{eqnarray*}
 & &	\ga^1_R=\{x\in (-R, R)\}, \ga^2_R=\{R+i h, h\in [0, 2\pi]\}, \\
 	& & 	\ga^3_R=\{x+2\pi i, x \in (R, -R)\}, \ga^4_R=\{-R+i h, h\in [2\pi , 0]\}.
 	\end{eqnarray*}
 	
 	\begin{center}
 		\begin{figure}
 			\label{fig:1} 
 			\includegraphics{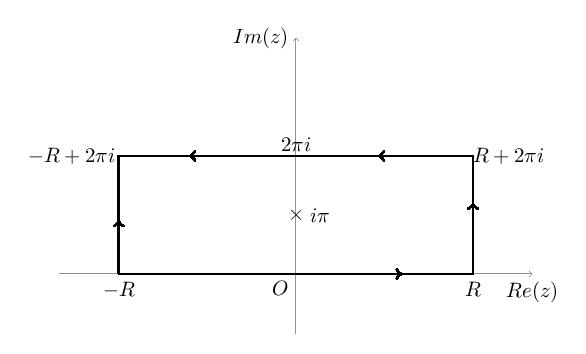}
 			\caption{Contour of integration}
 		\end{figure}
 	\end{center}

 	Clearly, 
 	$$
 	\int_{\ga^1_R} \f{e^{a z}}{(e^z+1)^j} dz+	\int_{\ga^3_R} \f{e^{a z}}{(e^z+1)^j} dz=
 	(1-e^{2\pi a i}) \int_{-R}^R \f{e^{t a}}{(e^t+1)^j} dt, 
 	$$
 	while for $R>>1$, 
 	$$
 	\left| 	\int_{\ga^2_R} \f{e^{a z}}{(e^z+1)^j} dz\right|\leq C \f{e^{R a}}{(e^R-1)^j}, \ \ 
 	\left| 	\int_{\ga^4_R} \f{e^{a z}}{(e^z+1)^j} dz\right|\leq C e^{-a R}. 
 	$$
 	It follows that 
 	$$
 	\lim_{R\to \infty} \int_{\ga_R} \f{e^{a z}}{(e^z+1)^j} dz=(1-e^{2\pi a i}) \int_{-\infty}^\infty \f{e^{t a}}{(e^t+1)^j} dt.
 	$$
 	It remains to compute the residues associated with this complex integration. This is a straightforward calculation, the results of which are below 
 	\begin{eqnarray}
 		\label{s:26}
 		Res\left(\f{e^{a z}}{e^z+1}, \pi i \right) &=&-  e^{i a \pi}  \\ 
 		\label{s:24}
 	 	Res\left(\f{e^{a z}}{(e^z+1)^2}, \pi i \right) &=&-(1-a)  e^{i a \pi}   \\ 
 	\label{s:21}
 	Res\left(\f{e^{a z}}{(e^z+1)^3}, \pi i \right) &=& -\f{1}{2} (2-a)(1-a) e^{i a \pi}
 	\end{eqnarray}
 The formulas \eqref{b:12}, \eqref{b:14}, \eqref{b:16} follow by substituting these expressions in the residue formulas and taking $R\to \infty$. 
 \end{proof}


\end{document}